\documentclass[11pt,reqno]{amsart}
\usepackage[T1]{fontenc}
\usepackage{lmodern}
\usepackage{amsmath,amssymb,mathtools,mathrsfs}
\usepackage[margin=1in]{geometry}
\usepackage{microtype}
\usepackage{needspace}
\usepackage{placeins}
\usepackage{tikz-cd}
\usepackage{todonotes}

\usetikzlibrary{arrows.meta,calc,fit}
\tikzset{
 cell/.style={circle,draw,fill=white,inner sep=0pt,minimum size=4.5pt},
 attachment/.style={thin},
 cellmap/.style={-{Stealth[length=1.8mm]},semithick},
 cellframe/.style={draw,rounded corners=2pt,thin,inner sep=4pt}
}
\usepackage{xcolor}
\usepackage[colorlinks=true,linkcolor=blue!45!black,citecolor=blue!45!black,urlcolor=blue!45!black]{hyperref}
\newtheorem{theorem}{Theorem}[section]
\newtheorem{maintheorem}{Theorem}

\AddToHook{env/maintheorem/before}{\stepcounter{theorem}}
\newtheorem{proposition}[theorem]{Proposition}
\newtheorem{conjecture}[theorem]{Conjecture}
\newtheorem{lemma}[theorem]{Lemma}
\newtheorem{corollary}[theorem]{Corollary}
\theoremstyle{definition}
\newtheorem{definition}[theorem]{Definition}
\newtheorem{example}[theorem]{Example}
\theoremstyle{remark}
\newtheorem{remark}[theorem]{Remark}
\newtheorem*{remark*}{Remark}
\newcommand{\Z}{\mathbb Z}

\newcommand{\F}{\mathbb F}

\newcommand{\Sp}{\mathrm {Sp}}
\newcommand{\D}{\mathcal D}
\newcommand{\Id}{\mathrm{id}}
\newcommand{\Res}{\operatorname{Res}}
\newcommand{\Ind}{\operatorname{Ind}}
\newcommand{\Hom}{\operatorname{Hom}}

\newcommand{\Inf}{\operatorname{Inf}}

\newcommand{\cofib}{\operatorname{cofib}}
\newcommand{\Cof}{\operatorname{Cof}}

\newcommand{\RO}{RO}
\newcommand{\Fun}{\operatorname{Fun}}
\newcommand{\Ho}{\operatorname{Ho}}

\numberwithin{equation}{section}
\hypersetup{pdftitle={Equivariant generating hypotheses for finite groups},pdfauthor={Sihao Ma, XiaoLin Danny Shi, Zhouli Xu, Shangjie Zhang},pdfsubject={Ghosts for every nontrivial finite group}}

\newcommand{\eulercells}[5]{
\begin{tikzpicture}[x=1cm,y=1cm,scale=.94,transform shape,font=\small,
 dimcell/.style={circle,draw,fill=white,minimum size=.64cm,inner sep=.6pt,font=\tiny},
 moore/.style={rectangle,minimum width=1.34079cm,minimum height=2.2cm,inner sep=0pt}]
 \foreach \x/\n in {0/s,5.5/t} {
  \node[dimcell] (\n0) at (\x-0.08031,-0.03539) {$#2$};
  \node[dimcell] (\n1) at (\x-0.53539,1.21493) {$#2+1$};
  \node[dimcell] (\n2) at (\x+0.76796,2.32913) {$1$};
  \node[dimcell] (\n3) at (\x+0.31288,3.57945) {$2$};
  \draw[attachment] (\n1) -- node[right=3pt] {$#1$} (\n0);
  \draw[attachment] (\n3) -- node[right=3pt] {$#1$} (\n2);
  \draw[attachment] (\n3) -- node[pos=.60,left=3pt,fill=white,inner sep=1pt] {$a_{#2}$} (\n1);
  \draw[attachment] (\n2) -- node[pos=.40,right=3pt,fill=white,inner sep=1pt] {$a_{#2}$} (\n0);
  \node[moore] (\n B) at (\x-0.30785,0.58977) {};
  \node[moore] (\n T) at (\x+0.54042,2.95429) {};
  \draw[rotate around={20:(\n B.center)}] (\n B.center) ellipse [x radius=.64cm,y radius=1.3cm];
  \draw[rotate around={20:(\n T.center)}] (\n T.center) ellipse [x radius=.64cm,y radius=1.3cm];
  \node at (\x-0.30785,-.95) {$#4$};
 }
 \draw[cellmap] (sT.east) -- node[above=3pt] {$0$} (tT.west);
 \draw[cellmap] (sB.east) -- node[above=3pt] {$#3$} (tB.west);
 \node[anchor=east] at (-1.8,1.62) {$#5:$};
\end{tikzpicture}}

\newcommand{\freecells}[6]{
\begin{tikzpicture}[x=1cm,y=1cm,scale=.94,transform shape,font=\small,
 dimcell/.style={rectangle,draw,fill=white,minimum size=.65cm,inner sep=0pt,font=\fontsize{8}{9}\selectfont},
 moore/.style={rectangle,minimum width=1.34079cm,minimum height=2.2cm,inner sep=0pt}]
 \foreach \x/\n in {0/s,5.5/t} {
  \node[dimcell] (\n0) at (\x-0.08031,-0.03539) {$1$};
  \node[dimcell] (\n1) at (\x-0.53539,1.21493) {$2$};
  \node[dimcell] (\n2) at (\x+0.76796,2.32913) {$2$};
  \node[dimcell] (\n3) at (\x+0.31288,3.57945) {$3$};
  \draw[attachment] (\n1) -- node[right=3pt] {$#1$} (\n0);
  \draw[attachment] (\n3) -- node[right=3pt] {$#1$} (\n2);
  \draw[attachment] (\n3) -- node[pos=.60,left=3pt,fill=white,inner sep=1pt] {$1-g$} (\n1);
  \draw[attachment] (\n2) -- node[pos=.40,right=3pt,fill=white,inner sep=1pt] {$1-g$} (\n0);
  \node[moore] (\n B) at (\x-0.30785,0.58977) {};
  \node[moore] (\n T) at (\x+0.54042,2.95429) {};
  \draw[rotate around={20:(\n B.center)}] (\n B.center) ellipse [x radius=.64cm,y radius=1.3cm];
  \draw[rotate around={20:(\n T.center)}] (\n T.center) ellipse [x radius=.64cm,y radius=1.3cm];
  \node at (\x-0.30785,-.95) {$#4$};
 }
 \draw[cellmap] (sT.east) -- node[above=3pt] {$#6$} (tT.west);
 \draw[cellmap] (sB.east) -- node[above=3pt] {$#3$} (tB.west);
 \node[anchor=east] at (-1.8,1.62) {$#5:$};
\end{tikzpicture}}

\title[Equivariant generating hypotheses]{Equivariant generating hypotheses for finite groups}
\author{Sihao Ma}
\address{Department of Mathematics, UCLA, Los Angeles, CA 90095-1555, USA}
\email{masihao@math.ucla.edu}
\author{XiaoLin Danny Shi}
\address{Department of Mathematics, University of Washington, Seattle, WA 98195, USA}
\email{dannyshi@uw.edu}
\author{Zhouli Xu}
\address{Department of Mathematics, UCLA, Los Angeles, CA 90095-1555, USA}
\email{xuzhouli@ucla.edu}
\author{Shangjie Zhang}
\address{Department of Mathematics, University of Washington, Seattle, WA 98195, USA}
\email{sjzh@uw.edu}
\date{}
\subjclass[2020]{Primary 55P91, Secondary 55P42, 55Q10}
\keywords{Equivariant generating hypothesis, finite equivariant spectra, ghosts}

\begin{document}
\begin{abstract}
We prove that no nontrivial finite group $G$ satisfies the equivariant generating hypothesis of Bohmann, even when all $RO(H)$-graded homotopy groups are tested for every subgroup $H\leq G$. For each $G$, we construct ghosts on finite $G$-spectra with arbitrarily long nonzero composition powers. We also prove that the homotopy-module functors are nonfull and construct non-equivalent finite $G$-spectra with isomorphic full homotopy modules.

A central ingredient of our proof is the natural ghost operation $\Psi=(p-[C_p/e])(1-g)$ on $C_p$-spectra. Motivated by Ma--Xu's examples in the motivic settings, we apply this operation to both a two-cell complex and products of complex projective spaces with cyclic permutation action, which produces arbitrarily long nonzero ghost powers. We also construct a ghost on an explicit four-cell complex that is detected in the derived category over the group ring.
\end{abstract}
\maketitle

\begingroup
\hypersetup{hidelinks}
\fontsize{10}{11}\selectfont
\tableofcontents
\endgroup

\section{Introduction}
Freyd's generating hypothesis asks whether stable homotopy groups detect maps between finite spectra~\cite{Freyd}.

\begin{conjecture}[Freyd's classical generating hypothesis]\label{conj:freyd}
Let $f:X\to Y$ be a map between finite spectra. If
\[
 \pi_n(f):\pi_nX\longrightarrow\pi_nY
\]
is zero for every $n\in\Z$, then $f=0$ in the stable homotopy category.
\end{conjecture}

The conjecture remains open. Freyd proved that the conjectured faithfulness would imply fullness~\cite[Proposition~9.7]{Freyd}: for finite $X$ and $Y$, the natural map
\[
 [X,Y]\longrightarrow\Hom_{\pi_*S}(\pi_*X,\pi_*Y)
\]
would be an isomorphism, where module homomorphisms preserve degrees. Graded homotopy modules would then recover both maps and stable homotopy types of finite spectra.

Bohmann formulated the equivariant generating hypothesis in 2010, using integer-graded homotopy groups at every subgroup \cite[Conjecture~1.2]{Bohmann}. In this paper, we disprove it for every nontrivial finite group. The conjecture remains false after testing all representation-graded homotopy groups at every subgroup, and the resulting ghost ideal is nonnilpotent for each fixed nontrivial finite group. We also prove that the homotopy-module functors are nonfull and construct non-equivalent finite $G$-spectra with isomorphic full homotopy modules.

\Needspace{8\baselineskip}
Fix a finite group $G$. We work with $G$-spectra indexed on a complete universe~\cite{LMS}. Finite $G$-spectra are the compact objects in this category. They form the thick subcategory generated by the orbit spectra $\Sigma^\infty_G G/H_+$, $H\leq G$~\cite[Section~2]{BS}. Write $[X,Y]^G$ for stable maps and put
\[
 \pi_W^H(X)=[S^W,\Res_H^G X]^H
 \qquad(H\leq G,\ W\in\RO(H)).
\]
The $H$-Mackey functor $\underline{\pi}_W^H(X)$ is defined on orbits by
\[
 \underline{\pi}_W^H(X)(H/L)
 = [S^{\Res_L^H W},\Res_L^G X]^L
 \qquad(L\leq H).
\]
We use $*$ for integer degrees and $\star$ for representation degrees. In particular, $\underline{\pi}_*X$ denotes the integer-graded $G$-Mackey functor, and $\underline{\pi}_\star^H(X)$ denotes the $\RO(H)$-graded family of $H$-Mackey functors. We consider the following three versions of the equivariant generating hypothesis.

\Needspace{8\baselineskip}
\begin{conjecture}[Equivariant generating hypotheses for finite groups]\label{conj:equivariant}
Let $G$ be finite and let $f:X\to Y$ be a map between finite $G$-spectra. Each of the following formulations claims that the corresponding vanishing condition implies $f=0$.
\begin{enumerate}
\item\label{item:integer}\textup{\textbf{Bohmann's integer-graded Mackey GH.}}
The condition is
\[
 \underline{\pi}_*(f)=0.
\]
\item\label{item:ro}\textup{\textbf{$\RO(G)$-graded Mackey GH.}}
The condition is
\[
 \underline{\pi}_\star^G(f)=0.
\]
\item\label{item:all}\textup{\textbf{All-subgroup representation-orbit GH.}}
The condition is that, for every $H\leq G$,
\[
 \underline{\pi}_\star^H(f)=0.
\]
\end{enumerate}
\end{conjecture}

By the induction--restriction adjunction, these conditions test maps from $G/H_+\wedge S^n$, $G/H_+\wedge S^V$, and $G_+\wedge_H S^W$, respectively, where $H\leq G$, $n\in\Z$, $V\in\RO(G)$, and $W\in\RO(H)$. All three families consist of compact objects and generate the finite $G$-spectra as a thick subcategory. They increase from parts~\textup{(\ref{item:integer})} to~\textup{(\ref{item:all})}, since ordinary suspensions are representation suspensions and
\[
 G/H_+\wedge S^V\simeq G_+\wedge_H S^{\Res_H^G V}.
\]
Therefore, a counterexample to part~\textup{(\ref{item:all})} disproves all three. A map satisfying the vanishing condition for one of these families is called a \emph{ghost} for that family~\cite{Christensen}. For part~\textup{(\ref{item:all})}, we use the term \emph{representation-orbit ghost}.

For finite $G$, Greenlees--May~\cite[Appendix~A]{GM} proved faithfulness and fullness of the integer-graded homotopy Mackey functor on all rational $G$-spectra. All three hypotheses therefore hold rationally. Bohmann~\cite[Theorem~1.7]{Bohmann} gave a counterexample on compact rational $S^1$-equivariant spectra in the sense of Conjecture~\ref{conj:equivariant}\textup{(\ref{item:integer})}.

For each of the three detector families, the homotopy-module functor records homotopy groups together with the operations induced by stable maps between the detecting spectra.

\vspace{0.1in}

\noindent\textbf{Convention.} Unless otherwise specified, for the rest of this paper $G$ is a nontrivial finite group and $p$ is a prime dividing $|G|$.

\vspace{0.1in}

Bohmann~\cite[Theorem~1.5]{Bohmann} proved that integer-graded Mackey GH for finite $G$ implies fullness over the graded sphere Green functor. In this paper, we prove nonfaithfulness and nonfullness for all three homotopy-module functors.

\begin{maintheorem}[Nonfaithfulness and nonfullness]\label{thm:main-results}
~\begin{enumerate}
\item\label{item:main-failure}
All three hypotheses in Conjecture~\ref{conj:equivariant} are false, both integrally and in the category of finite $p$-local $G$-spectra.
\item\label{item:main-nonfullness}
All three homotopy-module functors on finite $G$-spectra are nonfull, both integrally and $p$-locally.
\end{enumerate}
\end{maintheorem}

If $\ell$ is a prime not dividing $|G|$, each of the three $\ell$-local versions of Conjecture~\ref{conj:equivariant} is equivalent to the classical nonequivariant $\ell$-local generating hypothesis. This follows from the splitting of the $\ell$-local $G$-equivariant homotopy category \cite{BGS, Liu}.

The equivariant homotopy of a finite $G$-spectrum $X$ contains more information than the module $\pi_\star^G X$ over the equivariant stable homotopy ring $\pi_\star^G S$. The homotopy groups at all subgroups also carry restrictions, transfers, and all other operations induced by stable maps between the detecting spectra. Morphisms of homotopy modules commute with all these operations (see Section~\ref{sec:module-detection}). Even with all this structure, our proof of nonfullness gives the following result.

\begingroup
\begin{maintheorem}[Homotopy modules]
\label{cor:same-modules}
There exist non-equivalent finite $G$-spectra $A$ and $B$ that are $p$-local and whose homotopy Mackey functors admit, for all $H\leq G$ and $W\in\RO(H)$, isomorphisms
\[
 \underline{\pi}_W^H(A)\cong\underline{\pi}_W^H(B)
\]
compatible with all stable operations.
\end{maintheorem}
\par
\endgroup
{The nonfaithfulness in Theorem~\ref{thm:main-results}\textup{(\ref{item:main-failure})} already occurs both on a spectrum with four free cells and a spectrum with two fixed cells.} We construct these examples in Sections~\ref{sec:prime-order}--\ref{sec:two-cell-ghosts}.

We also construct ghosts with nonzero powers of arbitrary length for each finite group.

\begin{maintheorem}[Arbitrarily long ghost powers]\label{thm:long-main}
For every $m\geq0$ and $n\geq1$, there is a compact object $Y$ in the $p$-local $G$-equivariant stable homotopy category and an endomorphism $f:Y\to Y$ such that
\[
 \pi_W^H(f)=0\quad(H\leq G,\ W\in\RO(H)),
 \qquad p^m\cdot f^n\neq0.
\]
\end{maintheorem}
In fact, we can further show that $(S^{-W}\wedge\Res_H^G f)^H$ is null as a map of spectra for every $H\leq G$ and $W\in\RO(H)$.

{These ghost powers give simultaneous lower bounds for the flat, projective, and injective dimensions of all three homotopy modules on the same finite spectrum.}

\begingroup
\begin{maintheorem}[Homological consequences]\label{cor:homological}
For every integer $n\geq1$, there is a finite $p$-local $G$-spectrum $Y_n$ such that the homotopy modules
\[
 \underline{\pi}_*Y_n,\qquad
 \underline{\pi}_\star^G(Y_n),\qquad
 \{\underline{\pi}_\star^H(Y_n)\}_{H\leq G}
\]
each have flat, projective, and injective dimension at least $n$ in their respective $p$-local module categories. Consequently, each of these categories has infinite weak global dimension and infinite global dimension.
\end{maintheorem}
\par
\endgroup

\begingroup
\subsection*{Constructions and detections}
First, we note that a counterexample to any of the three generating hypotheses for a subgroup $H\leq G$ induces a counterexample to the corresponding hypothesis for $G$.  This is proved in Theorem~\ref{lem:induction}. In particular, the double-coset formula shows that induction preserves the ghost property, while restricting the induced map back to $H$ recovers the original map as a direct summand and therefore produces a nonvanishing ghost. Therefore, it suffices to construct counterexamples for the group $C_p$ for all primes $p$.

The constructions in this paper are partly motivated by the work of the first and third authors~\cite{MX} on the algebraic and motivic generating hypotheses. Their constructions produce finite ghosts with arbitrarily long nonzero composition powers. In particular, we draw inspiration from two general ideas in their work: natural operations that vanish on the detecting objects and projective-space models supporting long ghost powers. 

\vspace{0.2cm}

\noindent \textbf{(1) A natural cyclic ghost operation.} The first idea is the following naturality principle.

\begin{lemma}\label{lem:central-ghost}
Let $\mathcal T$ be an additive category with a collection $\mathcal P$ of objects. Suppose $\Theta:\Id_{\mathcal T}\Rightarrow\Id_{\mathcal T}$ is natural and $\Theta_P=\Id_P$ for every $P\in\mathcal P$. Then $1-\Theta_X$ induces zero on $\Hom_{\mathcal T}(P,X)$ for every $P\in\mathcal P$ and $X\in\mathcal T$.
\end{lemma}
\begin{proof}
For $a:P\to X$, naturality gives $(1-\Theta_X)a=a-a\Theta_P=0$.
\end{proof}

Specializing to the $C_p$-equivariant homotopy category, we will show that the action by a chosen generator $g$ of $C_p$ for odd $p$ provides a natural choice of $\Theta$ for the collection $\mathcal{P}=\{S^V: V\in RO(C_p)\}$ of all representation spheres. Extending this observation to a broader range of settings leads to the following general method of producing natural equivariant ghosts.

\Needspace{8\baselineskip}
\begin{proposition}\label{prop:natural-main}
Let $g$ be a generator of $C_p$, and let $g_X$ be its action on a $C_p$-spectrum $X$. The natural operation
\[
 \Psi_X=(p-[C_p/e])(1-g_X)
\]
is a representation-orbit ghost with null underlying map.
\end{proposition}

The factors serve complementary purposes: $p-[C_p/e]$, where $[C_p/e]\in A(C_p)$ is the class of the free $C_p$-orbit, vanishes on underlying spectra, while $1-g_X$ vanishes on $C_p$-geometric fixed points. Their product therefore vanishes on representation spheres, and naturality makes $\Psi_X$ a ghost. It remains to find a finite $X$ with $\Psi_X\ne0$.

\vspace{0.2cm}

\noindent \textbf{(2) A nontrivial ghost on a four-cell complex.} 
We first construct a ghost on a complex consisting of four free cells, in analogy with the four-cell example in \cite[Definition~3.2]{MX}. Since $[C_p/e]$ is annihilated by the Euler class and restricts to multiplication by $p$, we are led to consider the equivariant map in Lemma~\ref{lem:cellular-power} on the cofiber of the Euler class, which is a complex with two free cells. Smashing with the mod-$p$ Moore spectrum, we obtain, for odd $p$, a ghost $f_p$ that turns out to satisfy the slightly stronger vanishing property
\[
 \bigl(S^{-V}\wedge \Res^{C_p}_H f_p\bigr)^H=0
 \qquad \ H= C_p\ \mathrm{or}\ e, \ V\in\RO(H).
\]
To detect the nontriviality of $f_p$, we consider the composite functor
\[
 \Sp^{C_p}\longrightarrow\Fun(BC_p,\Sp)
 \xrightarrow{H\F_p\wedge-}\Fun(BC_p,\mathcal D(\F_p))
 \simeq\mathcal D(\F_p[C_p]).
\]
The ghost induces a nonzero map in the derived category of the group ring $\F_p[C_p]$, while it becomes null after forgetting the $C_p$-action. At $p=2$, the construction uses the mod-$4$ Moore spectrum, and the detection uses $\Z/4$-coefficients.

We remark that the existence of a counterexample to the generating hypothesis often points to a global relation in the stable stems that cannot be explained solely by commutativity or shuffling of Toda brackets. In our case, this four-cell ghost corresponds to the Toda bracket relation in Proposition~\ref{lem:secondary-divisibility}.

\vspace{0.2cm}

\noindent \textbf{(3) A nontrivial ghost on a two-cell complex.} We then illustrate the effectiveness of Proposition~\ref{prop:natural-main} by producing a nontrivial $C_p$-equivariant ghost on a two-cell complex for each prime $p$. For $p=2$, we consider the map $\Psi_X$ for 
\[X:=\Cof(\eta\sigma+\epsilon[C_2/e])\]
with $\eta\sigma, \epsilon\in \pi_8 S^0$ viewed as elements in $\pi_8^{C_2}S^0$ via the inflation functor $\Inf_e^{C_2}$. This $X$ is a special case of our more general construction. For odd $p$, the observation above Proposition~\ref{prop:natural-main} implies that $1-g$ is already a ghost on $\Cof([C_p/e])$. 

In both cases, the map factors through a map from the top cell of the source to the bottom cell of the target. These factorizations allow us to detect the nontriviality using the inflation $\Inf_e^{C_p}H\Z$ for odd $p$, and by tom Dieck splitting and a direct computation in the equivariant stable homotopy ring $\pi_*^{C_2}S^0$ for $p=2$.

\vspace{0.2cm}

\noindent \textbf{(4) Arbitrarily long ghosts on projective spaces.} 
In \cite{MX}, the projective-space examples realized as finite cellular motivic spectra admit long ghost powers through the Gheorghe--Wang--Xu comparison~\cite{GWX}. These examples suggest $\mathbb C P^\infty$ as a natural object on which to seek nontrivial ghosts. To introduce a natural $C_p$-action, we use its $p$-fold product $(\mathbb C P^\infty)^p$, equipped with cyclic permutation of the factors. Using Proposition~\ref{prop:natural-main}, we obtain a natural ghost $\Psi$ on $\Sigma^\infty_+(\mathbb C P^\infty)^p$. 

To detect the non-triviality of powers of $\Psi$ on $\Sigma^\infty_+(\mathbb C P^\infty)^p \simeq \Sigma^\infty_+(BS^1)^p$, we consider a $C_p$-equivariant group embedding $C_{p^2}\hookrightarrow (S^1)^p$ and its induced map
\[\varphi: \Sigma^\infty_+ BC_{p^2} \to \Sigma^\infty_+(BS^1)^p.\]
Since $C_{p^2}$ is finite, this allows us to use the May–Snaith–Zelewski splitting \cite{MSZ} of the $p$-completed mapping spectrum 
\[
 F\left(\Sigma^\infty_+BC_{p^2},\Sigma^\infty_+(BS^1)^p\right)^\wedge_p\simeq \bigvee_{\substack{L\le C_{p^2}\\\rho:L\to(S^1)^p}}
 (\Sigma^\infty_+BW_\rho)^\wedge_p.\]
The key point is to upgrade this splitting to a $C_p$-equivariant splitting, which permutes chosen summands $\bigvee_{j=0}^{p-1}(\Sigma^\infty_+BW_{g^j\varphi})^\wedge_p$. Collapsing each summand to the sphere, we obtain a map 
\[
 \theta: F\left(\Sigma^\infty_+BC_{p^2},\Sigma^\infty_+(BS^1)^p\right)^\wedge_p\to \bigvee_{j=0}^{p-1}(\Sigma^\infty_+BW_{g^j\varphi})^\wedge_p\to \bigvee_{j=0}^{p-1} S^\wedge_p.
\]
As shown in Theorem~\ref{thm:nonvanishing}, the induced map of $\theta$ on $\pi_0^{C_p}$ takes the composite
\[\Psi^n\varphi: \Sigma^\infty_+BC_{p^2} \xrightarrow{\varphi}\Sigma^\infty_+(BS^1)^p \xrightarrow{\Psi^n} \Sigma^\infty_+(BS^1)^p\]
to a nontrivial element in $(\pi_0^{C_p} S_p^\wedge) [C_p]$ for any $n\geq 1$. Finally, for a fixed $n$, we prove that $\Psi^n$ remains nontrivial after passing to finite skeleton $\Sigma^\infty_+(\mathbb C P^k)^p$ for sufficiently large $k$. 

Using the induction principle of Theorem~\ref{lem:induction}, we obtain Theorem~\ref{thm:long-main} for all nontrivial finite groups.

\vspace{0.2cm}

\noindent \textbf{(5) Comparisons.} 
Table~\ref{table:comparison} below summarizes the constructions and detections of the examples in this paper.

\begin{table}[htbp]
  \centering
  \setlength{\tabcolsep}{5pt}
  \renewcommand{\arraystretch}{1.4}
  \begin{tabular}{@{}lccc@{}}
    \hline
    & Complex & Ghost & Detection Method \\
    \hline
    Four-cell, $p$ odd
    & $\Cof(a_\lambda)\wedge S/p$
    & $1-q_{p+1}$
    & $\mathcal{D}(\F_p[C_p])$ \\
    Four-cell, $p=2$
    & $\Cof(a_{2\sigma})\wedge S/4$
    & $2(1-q_3)$
    & $\mathcal{D}(\Z/4[C_2])$ \\
    Two-cell, $p$ odd
    & $\Cof([C_p/e])$
    & $1-g$
    & $\Inf_e^{C_p}H\Z$ \\
    Two-cell, $p=2$
    & $\Cof(\eta\sigma+\epsilon[C_2/e])$
    & $\Psi_X$
    & tom Dieck splitting and $\pi_*^{C_2}S^0$  \\
    Arbitrarily long ghost
    & $\Sigma^\infty_+(\mathbb C P^k)^p$
    & $\Psi_X$
    & May--Snaith--Zelewski splitting\\
    \hline
  \end{tabular}
  \vspace{0.2cm}
  \caption{Comparison of the ghost constructions.}
  \label{table:comparison}
\end{table}

\subsection*{Organization}
Sections~\ref{sec:prime-order}--\ref{sec:all-groups} prove Theorem~\ref{thm:main} for the four-cell ghost, which implies Theorem~\ref{thm:main-results}\textup{(\ref{item:main-failure})}. Section~\ref{sec:prime-order} constructs the $C_p$-examples from circle power maps and explains the intuition behind the construction. Section~\ref{sec:prime-order-proofs} proves that the maps vanish after every representation twist and passage to fixed points, and detects their nonvanishing using chains with $C_p$-action. Section~\ref{sec:all-groups} passes to arbitrary finite groups by the induction principle.

Section~\ref{sec:two-cell-ghosts} proves Proposition~\ref{prop:natural-main} and provides the two-cell examples. Section~\ref{sec:long-ghosts} proves
Theorem~\ref{thm:long-main} by considering the natural ghost on $(\mathbb CP^\infty)^p$ and then passing to finite complexes and arbitrary finite groups. Section~\ref{sec:module-detection} uses double ghosts to prove Theorem~\ref{thm:main-results}\textup{(\ref{item:main-nonfullness})} and distinguishes spectra with isomorphic homotopy modules to prove Theorem~\ref{cor:same-modules}. Section~\ref{sec:homological} proves the homological dimension bounds of Theorem~\ref{cor:homological} from arbitrarily long ghost powers.

\subsection*{Acknowledgments}
The second author is partially supported by NSF Grant No.~DMS-2404828 and NSF CAREER Grant No.~DMS-2541934. The third author is partially supported by the AMS Centennial Research Fellowship and NSF Grant DMS-2506247. The fourth author gratefully acknowledges that this research is supported in part by the Pacific Institute for the Mathematical Sciences.

\section{\texorpdfstring{A four-cell ghost for $C_p$}{A four-cell ghost for Cp}}\label{sec:norm}\label{sec:prime-order}

In Sections~\ref{sec:prime-order}--\ref{sec:all-groups}, we will prove the following result: 

\begin{theorem}[Four-cell ghosts]\label{thm:main}
There is a finite $G$-spectrum $X$ with four free $G$-orbit cells and an endomorphism $f:X\to X$ of additive order $p$ such that, for every $H\leq G$ and every $W\in\RO(H)$, the map
\begin{equation}\label{eq:null}
 \bigl(S^{-W}\wedge\Res_H^G f\bigr)^H:
 \bigl(S^{-W}\wedge\Res_H^G X\bigr)^H
 \longrightarrow
 \bigl(S^{-W}\wedge\Res_H^G X\bigr)^H
\end{equation}
is null, but its image $f:X\to X$ in $\Fun(BG,\Sp)$ is nonzero.
\end{theorem}
Here, $\Fun(BG,\Sp)$ is the category of Borel $G$-spectra, consisting of ordinary spectra with homotopy coherent $G$-action. The fact that map~\eqref{eq:null} is null is a stronger condition than vanishing on homotopy groups. Moreover, this four-cell ghost can be constructed with $X$ already $p$-local and $f^2=0$.

In this section, we will first construct the four-cell ghost for $C_p$.  Once we have constructed this ghost for $C_p$, section~\ref{sec:all-groups} will extend the construction to all nontrivial finite groups by induction.

\subsection{A four-cell counterexample}\label{subsec:power-maps}
For $r>1$, let
\[
 S/r=\cofib(r:S\longrightarrow S)
\]
denote the mod-$r$ Moore spectrum.  Let $C_p=\langle g\rangle$ denote the cyclic group of order $p$ for some prime $p$, and let $\lambda$ be the real representation underlying the complex character $g\mapsto e^{2\pi i/p}$. Let
\[
 a_{\lambda}:S^0\longrightarrow S^{\lambda}
\]
denote the Euler class of $\lambda$. There is a cofiber sequence
\begin{equation}\label{eq:euler-triangle}
 S^0\xrightarrow{a_{\lambda}}S^{\lambda}\xrightarrow{i}\Cof(a_{\lambda})
       \xrightarrow{\delta}S^1,
\end{equation}
where $\Cof(a_\lambda)$ can be identified with $\Sigma S(\lambda)_+$.

Recall the Burnside ring of $C_p$
\begin{equation}\label{eq:burnside-ring}
 A(C_p)=\mathbb Z\{1,[C_p/e]\},
\end{equation}
where \[
 [C_p/e]^2=p[C_p/e], \qquad a_\lambda\cdot [C_p/e]=0.
\]
The marks homomorphism gives an injection
\begin{equation}\label{eq:marks}
A(C_p)\hookrightarrow\mathbb Z^2,
 \qquad x+y[C_p/e]\longmapsto(x+py,x).
\end{equation}

Moreover let
\[ N:=1+g+\cdots+g^{p-1}\in\Z[C_p]\]
be the norm in the group ring. It satisfies
\[
(1-g)N=0,\qquad N^2=p\cdot N.
\]

Consider the map 
\begin{align*}
    q_{p+1}: S(\lambda) &\to S(\lambda)\\
        z&\mapsto z^{p+1}.
\end{align*}
In particular, it is $C_p$-equivariant at the space level. Abusing the notation, we let $q_{p+1}$ also denote its {stabilization} on $\Cof(a_{\lambda})$.

\begin{lemma}\label{lem:cellular-power}
There is a commutative diagram of cofiber sequences:
\[
 \begin{tikzcd}[column sep=large,row sep=large]
 S^0 \arrow[r,"a_{\lambda}"] \arrow[d,"0"'] &
 S^{\lambda} \arrow[r,"i"] \arrow[d,"{-[C_p/e]}"'] &
 \Cof(a_{\lambda}) \arrow[r,"\delta"] \arrow[d,"1-q_{p+1}"] &
 S^1 \arrow[d,"0"] \\
 S^0 \arrow[r,"a_{\lambda}"'] & S^{\lambda} \arrow[r,"i"'] &
 \Cof(a_{\lambda}) \arrow[r,"\delta"'] & S^1.
 \end{tikzcd}
\]
In particular,
\begin{equation}\label{eq:power-relations}
 \begin{aligned}
 (1-q_{p+1})i&=- i [C_p/e],\\
 \delta(1-q_{p+1})&=0.
 \end{aligned}
\end{equation}
\end{lemma}
\begin{proof}
Consider the standard cofiber sequence
\[
 S(\lambda)_+\longrightarrow S^0\xrightarrow{a_\lambda}S^\lambda.
\]
The map $q_{p+1}$ on $S(\lambda)$ and the identity on $S^0$ induce a map on its mapping cone $S^\lambda$. This map is the one-point compactification of the radial extension of $q_{p+1}$, which we denote by
\[
 \begin{aligned}
 F_{p+1}&:\lambda\longrightarrow\lambda,\\
 F_{p+1}(re^{i\theta})&=re^{i(p+1)\theta},\\
 F_{p+1}(0)&=0.
 \end{aligned}
\]
It is proper and $C_p$-equivariant. Taking one-point compactification, $\widehat F_{p+1}$ fixes $0$ and $\infty$, so
\[
 \widehat F_{p+1}a_\lambda=a_\lambda.
\]
Using the identification $[S^\lambda,S^\lambda]^{C_p}\cong A(C_p)$, we compute its stable class from its degrees on underlying and $C_p$-fixed points. The underlying map has degree $p+1$, since $q_{p+1}$ winds the circle $p+1$ times. On $C_p$-fixed points, $S^{\lambda^{C_p}}=S^0$, and $\widehat F_{p+1}$ is the identity. These degrees are $p+1$ and $1$, respectively. Since the map $A(C_p)\hookrightarrow\mathbb Z^2$ above is injective, we obtain
\[
 [\widehat F_{p+1}]=1+[C_p/e].
\]
\par

Rotating the cofiber sequence gives a map of~\eqref{eq:euler-triangle} with components $1$, $\widehat F_{p+1}$, $q_{p+1}$, and $1$, where the third component is $\Sigma(q_{p+1})_+$ under $\Cof(a_\lambda)\simeq\Sigma S(\lambda)_+$. Consequently, the middle and right squares give, respectively,
\[
 \begin{aligned}
 q_{p+1}i&=i(1+[C_p/e]),\\
 \delta q_{p+1}&=\delta.
 \end{aligned}
\]
Subtracting this map of cofiber sequences from the identity proves the two asserted relations:
\[
 \begin{aligned}
 (1-q_{p+1})i&=i-i(1+[C_p/e])=-i[C_p/e],\\
 \delta(1-q_{p+1})&=\delta-\delta=0.
 \end{aligned}
\]
\end{proof}

\Needspace{8\baselineskip}

 We are now ready to introduce our four-cell model for the counterexample.

\begin{definition}\label{eq:small-models}

For $p$ odd, define
\[X_p: =\Cof(a_{\lambda})\wedge S/p,\]
\[f_p =1-q_{p+1}: X_p \to X_p.\]
For $p=2$, let $\sigma$ denote the $C_2$-sign representation, and define
\[X_2: =\Cof(a_{2\sigma})\wedge S/4,\]
\[f_2=2\cdot(1-q_3): X_2\to X_2.\]
\end{definition}
{Figure~\ref{fig:euler-cells} shows cell diagrams, as in \cite{WX}, of $X_p$ and $X_2$. Circles denote cells, with the $RO(G)$-degree of the corresponding representation spheres labeled inside. Each oval containing two cells denotes a Moore spectrum. Lines without arrows denote attaching maps, labeled alongside, and horizontal arrows denote morphisms.}

\begin{figure}[!htbp]
\centering
\eulercells{\hspace{-3pt}p}{\lambda}{-[C_p/e]}{X_p}{f_p}

\vspace{3mm}
\eulercells{\hspace{-4pt}4}{2\sigma}{-2\cdot [C_p/e]}{X_2}{f_2}
\caption{A cell diagram for $X_p$ and $X_2$.}
\label{fig:euler-cells}
\end{figure}

We also note the following.
\begin{lemma}\label{lem:free-cell-power}
There is a cofiber sequence of free cells: 
\begin{equation}\label{eq:circle-triangle}
 \Sigma(C_p)_+\xrightarrow{1-g}\Sigma(C_p)_+
 \xrightarrow{j}\Cof(a_{\lambda})\xrightarrow{\partial}\Sigma^2(C_p)_+.
\end{equation}
Here $g$ acts by right translation. Moreover, there is a commutative diagram of cofiber sequences:
\[
\begin{tikzcd}[column sep=large,row sep=large]
 \Sigma(C_p)_+ \arrow[r,"1-g"] \arrow[d,"-\Sigma N"'] &
 \Sigma(C_p)_+ \arrow[r,"j"] \arrow[d,"0"'] &
 \Cof(a_\lambda) \arrow[r,"\partial"] \arrow[d,"1-q_{p+1}"] &
 \Sigma^2(C_p)_+ \arrow[d,"-\Sigma^2N"] \\
 \Sigma(C_p)_+ \arrow[r,"1-g"'] &
 \Sigma(C_p)_+ \arrow[r,"j"'] &
 \Cof(a_\lambda) \arrow[r,"\partial"'] & \Sigma^2(C_p)_+.
\end{tikzcd}
\]
In particular,
\begin{equation}\label{eq:free-power-relations}
 (1-q_{p+1})j=0,\qquad
 \partial(1-q_{p+1})=-(\Sigma^2N)\partial.
\end{equation}
\end{lemma}
\begin{proof}
For the free-cell structure and attaching map $1-g$, see~\cite[Section~2.4]{BHZ}. The map $q_{p+1}$ fixes the vertices, so $q_{p+1}j=j$. It sends the chosen oriented edge to itself followed by one full turn, giving $1+N$ on the relative one-cells. The induced map of cofiber sequences therefore has components $1+\Sigma N$, $1$, $q_{p+1}$, and $1+\Sigma^2N$. In particular, $\partial q_{p+1}=(1+\Sigma^2N)\partial$. Subtracting from the identity gives the diagram and both relations.
\end{proof}

In particular, $X_p$ has an alternative cell-structure consisting of only free-cells, as shown in Figure~\ref{fig:freecells}. Each square labeled with $n$ denotes a free cell $\Sigma^n {C_p}_+$.

\begin{figure}[!htbp]
\centering
\freecells{\hspace{-3pt}p}{C_p}{0}{X_p}{f_p}{-N}

\vspace{3mm}
\freecells{\hspace{-4pt}4}{C_2}{0}{X_2}{f_2}{-2N}
\caption{The free-cell structure for $X_p$ and $X_2$.}
\label{fig:freecells}
\end{figure}

Since $\Phi^e(\Cof(a_\lambda))$ splits as $\Sigma S\vee\Sigma^2S$, we have
\begin{equation}\label{eq:underlying-power}
 \Phi^e(q_{p+1})=\begin{pmatrix}1&0\\0&p+1\end{pmatrix}.
\end{equation}

\subsection{Proof of Theorem~\ref{thm:main} for $C_p$}\label{subsec:prime-proof}

The proof relies on two propositions, whose proofs will be given in Section~\ref{sec:prelim} and Section~\ref{sec:chains}, respectively.

The first proposition claims that $f_p$ induces trivial maps on the $H$-fixed points of the mapping spectrum $\operatorname{Map}(S^V,\Res_H^{C_p}X_p)$, and therefore induces trivial maps on $\pi_V^H$ for every $H\leq C_p$ and $V\in\RO(H)$.

\begin{proposition}\label{prop:twisted-nullity}
The power maps on the spectra in Definition~\ref{eq:small-models} satisfy the following statements.
\begin{enumerate}
\item For an odd prime $p$,
\[
 \bigl(S^{-V}\wedge\Res_H^{C_p}f_p\bigr)^H=0
 \qquad(H\leq C_p,\ V\in\RO(H)).
\]
\item At $p=2$,
\[
 \bigl(S^{-V}\wedge\Res_H^{C_2}f_2\bigr)^H=0
 \qquad(H\leq C_2,\ V\in\RO(H)).
\]
\end{enumerate}
\end{proposition}

{The second proposition detects $f_p$ on ordinary homology spectra while retaining the homotopy coherent $C_p$-action. The resulting map of chains is nonzero in the derived category of modules over the group ring, so $f_p$ itself is nonzero.}

\Needspace{8\baselineskip}
\begin{proposition}\label{lem:chains}\label{lem:two-term}
\leavevmode
\begin{enumerate}
\item Let $p$ be odd and let
\[
 C=\bigl(0\longrightarrow\F_p[C_p]\xrightarrow{1-g}
                    \F_p[C_p]\longrightarrow0\bigr)
\]
be concentrated in degrees one and zero. Define $\tau:C\longrightarrow C$ by
\[
 \tau_1=-(1+g+\cdots+g^{p-1}),\qquad \tau_0=0.
\]
Then, $H\F_p\wedge\Phi^e(X_p)$ with a residual $C_p$-action is represented by $\Sigma C\oplus\Sigma^2C$, and $H\F_p\wedge\Phi^e(f_p)$ is represented by $\Sigma\tau\oplus\Sigma^2\tau$. This map is nonzero in $\mathcal D(\F_p[C_p])$.
\item Let $p=2$ and let
\[
 C=\bigl(0\longrightarrow(\Z/4)[C_2]\xrightarrow{1-g}
                    (\Z/4)[C_2]\longrightarrow0\bigr)
\]
be concentrated in degrees one and zero. Define $\tau:C\longrightarrow C$ by
\[
 \tau_1=-2\cdot(1+g),\qquad \tau_0=0.
\]
Then, $H\Z/4\wedge\Phi^e(X_2)$ with a residual $C_2$-action is represented by $\Sigma C\oplus\Sigma^2C$, and $H\Z/4\wedge\Phi^e(f_2)$ is represented by $\Sigma\tau\oplus\Sigma^2\tau$. This map is nonzero in $\mathcal D((\Z/4)[C_2])$.
\end{enumerate}
In both cases, $f_p$ remains nonzero in the Borel category, although the induced map on ordinary homology groups is zero.
\end{proposition}

\begin{proof}[Proof of Theorem~\ref{thm:main} for $C_p$]
By Lemma~\ref{lem:free-cell-power}, $\Cof(a_\lambda)$ has two free $C_p$-orbit cells. Smashing with the two-cell Moore spectrum $S/p$ for odd $p$, or $S/4$ for $p=2$, gives four free orbit cells in $X_p$. Here $\lambda=2\sigma$ when $p=2$. By Proposition~\ref{prop:twisted-nullity},
\[
 \bigl(S^{-V}\wedge\Res_H^{C_p}f_p\bigr)^H=0
 \qquad(H\leq C_p,\ V\in\RO(H)).
\]
Taking $\pi_0$ gives $\pi_V^H(f_p)=0$ for every $H$ and $V$. Proposition~\ref{lem:chains} shows that $f_p$ is nonzero. The maps are therefore counterexamples to all three generating hypotheses.

The spectra are finite and already $p$-local, so they also give the $p$-local counterexamples. For odd $p$, Lemma~\ref{lem:moore}, smashed with $\Cof(a_\lambda)$, gives $p\cdot\Id_{X_p}=0$, so the additive order of $f_p$ is exactly $p$. At two, the same lemma gives $4\cdot\Id_{X_2}=0$. Together with Definition~\ref{eq:small-models}, this gives $2\cdot f_2=4\cdot(1-q_3)=0$, so the additive order of $f_2$ is exactly two.

\end{proof}

\subsection{A cell diagram interpretation}\label{subsec:euler-view}

We provide some intuitive explanations for Propositions~\ref{prop:twisted-nullity} and~\ref{lem:chains} in this section. The displayed cofiber sequence and Figures~\ref{fig:euler-divisibility-cells}--\ref{fig:toda-divisibility-cells} use the odd-prime notation. At $p=2$, replace $S/p$ by $S/4$, $\lambda$ by $2\sigma$, $b_\lambda$ by $\eta_{C_2}^{\,2}$, and $-[C_p/e]$ by $-2[C_2/e]$.

Consider the cofiber sequence~\eqref{eq:euler-triangle} smashed with $S/p$:
\[
 S/p\xrightarrow{a_\lambda}S^\lambda\wedge S/p
 \xrightarrow{i}X_p\xrightarrow{\delta}\Sigma S/p.
\]
For a map $\ell\in\pi_V^{C_p}X_p$, either its image $\delta\ell$ in the top Moore spectrum $\Sigma S/p$ is nonzero, or exactness makes $\ell$ factor through the bottom Moore spectrum $S^\lambda\wedge S/p$. First suppose that $\ell$ maps into the bottom part, that is, $\delta\ell=0$.

\begin{figure}[!htbp]
\centering
\begin{tikzpicture}[x=1cm,y=1cm,font=\small,
 moore/.style={rectangle,minimum width=.65cm,minimum height=.9cm,inner sep=0pt},
 sphere/.style={circle,draw,minimum size=.65cm,inner sep=1pt}]
 \foreach \x/\n in {0/s,2.8/t,10.0/r} {
  \node[moore] (\n B) at (\x,0) {};
  \node[moore] (\n T) at (\x,1.8) {};
  \draw (\n B.center) ellipse [x radius=.325cm,y radius=.45cm];
  \draw (\n T.center) ellipse [x radius=.325cm,y radius=.45cm];
  \node[above=4pt] at (\n T.north) {$\Sigma S/p$};
  \node[below=4pt] at (\n B.south) {$S^\lambda\wedge S/p$};
  \draw[attachment] (\n T.south) -- node[right=3pt] {$a_\lambda$} (\n B.north);
 }
 \node[sphere] (vl) at (-1.55,0) {$S^V$};
 \draw[cellmap] (vl.east) -- node[above=3pt] {$w$} (sB.west);
 \draw[cellmap] (sB.east) -- node[above=3pt] {$-[C_p/e]$} (tB.west);
 \node at (5.1,1.15) {Proposition~\ref{lem:euler-divisibility}};
 \node at (5.1,.7) {$\Longrightarrow$};
 \node[sphere] (vr) at (6.8,0) {$S^V$};
 \draw[cellmap] (vr.east) -- node[above=3pt] {$-[C_p/e]w=a_\lambda b_\lambda w$} (rB.west);
\end{tikzpicture}
\caption{{Maps factor through the bottom Moore spectrum (odd $p$).}}
\label{fig:euler-divisibility-cells}
\end{figure}

\FloatBarrier

\begin{proposition}\label{lem:euler-divisibility}
\leavevmode
\begin{enumerate}
\item Let $p$ be odd. There is a unique
$b_\lambda\in\pi_\lambda^{C_p}S_{(p)}$ with
\[
 a_{\lambda}b_\lambda=p-[C_p/e].
\]
After smashing with $S/p$, this gives $[C_p/e]=-a_{\lambda}b_\lambda$ as a map on $S^\lambda\wedge S/p$.
\item Let $p=2$. Then
\[
 a_{2\sigma}\eta_{C_2}^{\,2}=4-2\cdot [C_p/e].
\]
After smashing with $S/4$, this gives $2\cdot [C_p/e]=-a_{2\sigma}\eta_{C_2}^{\,2}$ as a map on $S^{2\sigma}\wedge S/4$. Here $\eta_{C_2}\in\pi_\sigma^{C_2}S$ has the convention $a_\sigma\eta_{C_2}=2-[C_p/e]$.
\end{enumerate}
\end{proposition}

\begin{proof}
For odd $p$, apply $[-,S_{(p)}]^{C_p}$ to~\eqref{eq:euler-triangle}. By Lemma~\ref{lem:free-cell-power} and the fact that $\pi_1S_{(p)}=\pi_2S_{(p)}=0$, we have the exact sequence
\[
 0\longrightarrow\pi_\lambda^{C_p}S_{(p)}
 \xrightarrow{a_\lambda}A(C_p)_{(p)}
 \xrightarrow{\operatorname{rank}}\Z_{(p)}.
\]
Since the kernel of the rank map is $\Z_{(p)}(p-[C_p/e])$, there exists $b_\lambda$ satisfying $a_\lambda b_\lambda=p-[C_p/e]$. Injectivity of $a_\lambda$ gives uniqueness. Smashing this identity with $S/p$ and using $p\cdot\Id_{S/p}=0$ from Lemma~\ref{lem:moore} gives $[C_p/e]=-a_\lambda b_\lambda$.

At $p=2$, the relation $a_\sigma^2\eta_{C_2}=2\cdot a_\sigma$ from~\cite[Section~2.2 and Lemma~4.2]{BXZ}, together with $a_\sigma\eta_{C_2}=2-[C_p/e]$ (see also~\cite[Section~8]{BC} for this normalization in the $2$-local setting), gives
\[
 a_{2\sigma}\eta_{C_2}^{\,2}
 =2\cdot a_\sigma\eta_{C_2}=4-2\cdot [C_p/e].
\]
Smashing with $S/4$ and using $4\cdot\Id_{S/4}=0$ from Lemma~\ref{lem:moore} gives the claimed factorization of $2\cdot [C_p/e]$.
\end{proof}

As a result, if $\ell:S^V\longrightarrow X_p$ has $\delta\ell=0$, choose $w$ with $\ell=iw$. For odd $p$, Proposition~\ref{lem:euler-divisibility}(1) gives
\[
 f_p\ell=-i[C_p/e]w=ia_\lambda b_\lambda w=0.
\]
For $p=2$, part~(2) gives
\[
 f_2\ell=-i(2\cdot [C_p/e])w=ia_{2\sigma}\eta_{C_2}^{\,2}w=0.
\]
{The induced map is zero, as illustrated in Figure~\ref{fig:euler-divisibility-cells}.}

Next, suppose that $\ell$ has nonzero image in the top Moore spectrum $\Sigma S/p$.

\begin{figure}[!htbp]
\centering
\begin{tikzpicture}[x=1cm,y=1cm,font=\small,
 moore/.style={rectangle,minimum width=.65cm,minimum height=.9cm,inner sep=0pt},
 sphere/.style={circle,draw,minimum size=.65cm,inner sep=1pt}]
 \foreach \x/\n in {0/s,2.8/t,10/r} {
  \node[moore] (\n B) at (\x,0) {};
  \node[moore] (\n T) at (\x,1.8) {};
  \draw (\n B.center) ellipse [x radius=.325cm,y radius=.45cm];
  \draw (\n T.center) ellipse [x radius=.325cm,y radius=.45cm];
  \node[above=4pt] at (\n T.north) {$\Sigma S/p$};
  \node[below=4pt] at (\n B.south) {$S^\lambda\wedge S/p$};
  \draw[attachment] (\n T.south) -- node[right=3pt] {$a_\lambda$} (\n B.north);
 }
 \node[sphere] (vl) at (-1.55,1.8) {$S^V$};
 \draw[cellmap] (vl.east) -- node[above=3pt] {$x$} (sT.west);
 \draw[cellmap] (sB.east) -- node[above=3pt] {$-[C_p/e]$} (tB.west);
 \node at (5.1,1.15) {Proposition~\ref{lem:secondary-divisibility}};
 \node at (5.1,.7) {$\Longrightarrow$};
 \node[sphere] (vr) at (7.5,1.8) {$S^V$};
 \draw[cellmap] (vr.south east) -- node[below left=2pt] {$a_\lambda z$} (rB.west);
\end{tikzpicture}
\caption{{Case of a nonzero image in the top Moore spectrum (odd $p$).}}
\label{fig:toda-divisibility-cells}
\end{figure}

\FloatBarrier

\begin{lemma}\label{lem:transfer-extension}
There is an extension
\[
 \bar t:\Cof(a_{\lambda})\longrightarrow S^\lambda,
 \qquad \bar t\,i=[C_p/e],
 \qquad 1-q_{p+1}=-i\bar t.
\]
\end{lemma}
\begin{proof}
By Lemma~\ref{lem:cellular-power}, $\delta(1-q_{p+1})=0$, so exactness gives $\bar t:\Cof(a_\lambda)\to S^\lambda$ with $-i\bar t=1-q_{p+1}$. The middle square gives $i(\bar t i-[C_p/e])=0$, so exactness shows that $\bar t i-[C_p/e]$ factors through $a_\lambda$ and has underlying degree zero. Also $(\bar t i-[C_p/e])a_\lambda=0$. Since $\Phi^{C_p}(a_\lambda)=1$, its $C_p$-fixed-point degree is zero. Injectivity of the mark homomorphism therefore gives $\bar t i=[C_p/e]$.
\end{proof}

Smash the extension in Lemma~\ref{lem:transfer-extension} with $S/p$ at odd primes and with $S/4$ at two. We fix the resulting maps
\[
 \bar t:X_p\longrightarrow S^\lambda\wedge S/p,
 \qquad
 \bar t_2:X_2\longrightarrow S^{2\sigma}\wedge S/4,
\]
which satisfy
\[
 \bar t i=[C_p/e],\qquad f_p=-i\bar t,
 \qquad
 \bar t_2i=[C_p/e],\qquad f_2=-i(2\cdot\bar t_2).
\]
For odd $p$ and $x\in\pi_{V-1}^{C_p}(S/p)$ with $a_\lambda x=0$, define the restricted Toda bracket~\cite[Definition~3.7]{CF}:
\[
 \langle [C_p/e],a_\lambda,x\rangle_{\bar t}
 :=\left\{
 \bar t\circ\widetilde x
 \;\middle|\;
 \begin{aligned}
 &\widetilde x:S^V\longrightarrow X_p,\\
 &\delta\circ\widetilde x=\Sigma x
 \end{aligned}
 \right\}.
\]
Here the extension $\bar t$ is fixed, while the lift $\widetilde x$ varies. At $p=2$, for $x\in\pi_{V-1}^{C_2}(S/4)$ with $a_{2\sigma}x=0$, we similarly define
\[
 \langle 2\cdot [C_p/e],a_{2\sigma},x\rangle_{2\cdot\bar t_2}
 :=\left\{
 (2\cdot\bar t_2)\circ\widetilde x
 \;\middle|\;
 \begin{aligned}
 &\widetilde x:S^V\longrightarrow X_2,\\
 &\delta\circ\widetilde x=\Sigma x
 \end{aligned}
 \right\}.
\]

{The following divisibility statement is a consequence of Proposition~\ref{prop:twisted-nullity}, whose proof appears in Section~\ref{sec:prime-order-proofs}.}

\Needspace{8\baselineskip}
\begin{proposition}\label{lem:secondary-divisibility}
\leavevmode
\begin{enumerate}
\item Let $p$ be odd. For every $V\in\RO(C_p)$ and
$x\in\pi_{V-1}^{C_p}(S/p)$ with $a_{\lambda}x=0$,
\[
 \langle [C_p/e],a_{\lambda},x\rangle_{\bar t}\subseteq a_{\lambda}\cdot\pi_V^{C_p}(S/p).
\]
\item Let $p=2$. For every $V\in\RO(C_2)$ and
$x\in\pi_{V-1}^{C_2}(S/4)$ with $a_{2\sigma}x=0$,
\[
 \langle 2\cdot [C_p/e],a_{2\sigma},x\rangle_{2\cdot\bar t_2}
 \subseteq a_{2\sigma}\cdot\pi_V^{C_2}(S/4).
\]
\end{enumerate}
\end{proposition}
\begin{proof}
For any lift $\widetilde x:S^V\longrightarrow X_p$, Proposition~\ref{prop:twisted-nullity} and the chosen extensions give
\[
 i\bar t\widetilde x=-f_p\widetilde x=0\quad(p\text{ odd}),
 \qquad i(2\cdot\bar t_2)\widetilde x=-f_2\widetilde x=0\quad(p=2).
\]
For odd $p$, smash the Euler cofiber sequence~\eqref{eq:euler-triangle} with $S/p$ and apply $[S^V,-]^{C_p}$. The resulting exact sequence contains
\[
 \pi_V^{C_p}(S/p)\xrightarrow{a_\lambda\cdot}
 [S^V,S^\lambda\wedge S/p]^{C_p}
 \xrightarrow{i_*}[S^V,X_p]^{C_p}.
\]
Since $i\bar t\widetilde x=0$, the class $\bar t\widetilde x$ lies in $\ker(i_*)=\operatorname{im}(a_\lambda\cdot)$. We can write $\bar t\widetilde x=a_\lambda z$ for some $z\in\pi_V^{C_p}(S/p)$. At $p=2$, the same argument using $S/4$ and $a_{2\sigma}$ gives $2\cdot\bar t_2\widetilde x=a_{2\sigma}z$ for some $z\in\pi_V^{C_2}(S/4)$.
\end{proof}

{Now suppose $\ell:S^V\longrightarrow X_p$ with} $x=\Sigma^{-1}\delta\ell\neq 0$. For odd $p$, Proposition~\ref{lem:secondary-divisibility}(1) gives $\bar t\ell=a_\lambda z$ for some $z\in\pi_V^{C_p}(S/p)$, so
\[
 f_p\ell=-i\bar t\ell=-ia_\lambda z=0.
\]
For $p=2$, part~(2) gives $2\cdot\bar t_2\ell=a_{2\sigma}z$ for some $z\in\pi_V^{C_2}(S/4)$, and
\[
 f_2\ell=-i(2\cdot\bar t_2\ell)=-ia_{2\sigma}z=0.
\]
{These two cases express the vanishing of $f_p$ on $\pi_V^{C_p}$ in terms of Euler divisibility. Figure~\ref{fig:toda-divisibility-cells} records the second case.}

The map $f_p$ itself is nontrivial. For odd $p$, smashing with $H\F_p$ in the $C_p$-Borel category splits the Moore factor as $H\F_p\wedge S/p\simeq H\F_p\vee\Sigma H\F_p$. The resulting map on equivariant chains is nonzero in $\D(\F_p[C_p])$, even though its map on ordinary homology is zero. For $p=2$, the analogous calculation uses $H\Z/4$.

\FloatBarrier

\section{Representation-twisted fixed points and nonvanishing}\label{sec:prime-order-proofs}

\subsection{Representation-twisted fixed points}\label{sec:prelim}

We prove Proposition~\ref{prop:twisted-nullity} in this subsection. We start by recalling the following classical results.

\begin{lemma}\label{lem:moore}
~\begin{enumerate}
\item For $p$ an odd prime, we have $\pi_1(S/p)=\pi_2(S/p)=0$, $[\Sigma S/p, S/p]=0$, and
\[
 p\cdot\Id_{S/p}=0.
\]
\item For $p=2$, we have \[4\cdot\Id_{S/4}=0.\]
\end{enumerate}
\end{lemma}
\begin{proof}
In the exact sequence
\[
 \pi_2S\xrightarrow{p}\pi_2S\to\pi_2(S/p) \to \pi_1S\xrightarrow{p}\pi_1S\to\pi_1(S/p)
 \to\pi_0S\xrightarrow{p}\pi_0S,
\]
since $\pi_1S=\pi_2S=\Z/2$ and $p$ is odd, multiplication by $p$ is an isomorphism on $\pi_1$ and $\pi_2$. Since multiplication by $p$ is injective, we have $\pi_2(S/p)=\pi_1(S/p)=0$. The long exact sequence
\[\pi_2(S/p)\to \pi_2(S/p) \to [\Sigma S/p, S/p]\to \pi_1(S/p)\to \pi_1(S/p)\]
then implies $[\Sigma S/p, S/p]=0$. For the order of the identity of Moore spectra, see \cite[p.~257]{Oka} for example.
\end{proof}

For $C_p=\langle g\rangle$ and $V\in\RO(C_p)$, we first compute the induced map of $g$ {after twisting by a representation sphere}.

\begin{lemma}\label{lem:free-twist}
For $V\in\RO(C_p)$, there is an equivalence
\begin{equation}\label{eq:free-twist}
 \bigl(S^{-V}\wedge(C_p)_+\bigr)^{C_p}\simeq S^{-|V|}.
\end{equation}
Under this equivalence, right translation by $g$ on $(C_p)_+$ corresponds to a nonequivariant map of spheres with the following degree:
\begin{enumerate}
\item For odd $p$, the degree is $1$.
\item For $p=2$, let $V=a+b\sigma$. Then the degree is $(-1)^b$.
\end{enumerate}
\end{lemma}

\begin{proof}
{Since $S^{-V}\wedge (C_p)_+$ is a free $C_p$-spectrum, \eqref{eq:free-twist} follows from the Adams isomorphism:}
\[
 (S^{-V}\wedge (C_p)_+)^{C_p}
 \simeq (S^{-V}\wedge (C_p)_+)/C_p
 \simeq S^{-|V|}.
\]

Since $g^p=e$, the degree $d$ of the action of $g$ satisfies $d^p=1$. In particular, $d$ must be $\pm 1$. If $p$ is odd, this forces $d=1$. For $p=2$ and $V=a+b\sigma$, the action on each trivial summand has degree $1$, whereas the action on $\sigma$ is the antipodal action and has degree $-1$. Its degree on $S^{-V}$ is therefore $(-1)^b$.
\end{proof}

For every $C_p$-spectrum $Z$, there is a natural equivalence
\begin{equation}\label{eq:moore-fixed}
 (S/r\wedge Z)^{C_p}\simeq S/r\wedge Z^{C_p},
\end{equation}
since smash-product and taking $C_p$-fixed points both preserve cofiber sequences.

\Needspace{5\baselineskip}

\begin{proof}[Proof of Proposition~\ref{prop:twisted-nullity}]
For $H=e$, by \eqref{eq:underlying-power}, the underlying map $1-q_{p+1}$ on $\Cof(a_\lambda)\simeq\Sigma S\vee\Sigma^2S$ can be identified with the matrix
\[
 (1-q_{p+1})^e=\begin{pmatrix}0&0\\0&-p\end{pmatrix}: \Sigma S\vee\Sigma^2S\to \Sigma S\vee\Sigma^2S.
\]
For odd $p$, smashing with $S/p$ therefore identifies $(f_p)^e$ with
\[
 \begin{pmatrix}0&0\\0&-p\cdot\Id_{\Sigma^2S/p}\end{pmatrix}.
\]
By~Lemma~\ref{lem:moore}(1), $p\cdot\Id_{S/p}=0$, so $f_p$ is null nonequivariantly. For $p=2$, the map is \[(f_2)^e=(2(1-q_3))^e=\begin{pmatrix}0&0\\0&-4\cdot \Id_{\Sigma^2 S/4}\end{pmatrix}\] 
which is zero by~Lemma~\ref{lem:moore}(2). Therefore, smashing with $S^{-|V|}$ is null as well.

For $H=C_p$, by~\eqref{eq:moore-fixed}, we may consider the map on fixed points before smashing with the Moore spectrum.
\begin{enumerate}
\item Suppose $p$ is odd. Applying Lemma~\ref{lem:free-twist}(1)
to the cofiber sequence in Lemma~\ref{lem:free-cell-power}, we have $(S^{-V}\wedge \Cof(a_\lambda))^{C_p}\simeq S^{1-|V|}\vee S^{2-|V|}$. Since each right translation $g^i$ has degree one, the norm $N$ acts by $p$. Lemma~\ref{lem:free-cell-power} now gives
\begin{equation}\label{eq:twist-matrix}
 \bigl(S^{-V}\wedge(1-q_{p+1})\bigr)^{C_p}
 =\begin{pmatrix}0&\beta\\0&-p\end{pmatrix},
\end{equation}
for some $\beta\in \pi_1 S=\mathbb Z/2$. After smashing with $S/p$, multiplication by $-p$ is zero by Lemma~\ref{lem:moore}. As $[\Sigma S/p, S/p]=0$ in Lemma~\ref{lem:moore}, $\beta \wedge \Id_{S/p}$ is null as well.

\item Let $p=2$ and $V=a+b\sigma$. If $b$ is even,
Lemma~\ref{lem:free-twist}(2) and~Lemma~\ref{lem:free-cell-power} give the same matrix form as~\eqref{eq:twist-matrix}. Then $2\beta=0$ plus $4\cdot\Id_{S/4}=0$ implies that $(S^{-V} \wedge f_2)^{C_2}$ is null.

If $b$ is odd, Lemma~\ref{lem:free-twist}(2) identifies\[ (S^{-V}\wedge \Cof(a_{2\sigma}))^{C_2} \simeq \Sigma^{1-|V|}S/2.\] By Lemma~\ref{lem:free-cell-power}, $1-q_3$ vanishes on the bottom cell, so it factors as
\[
 \Sigma^{1-|V|}S/2\xrightarrow{\partial}S^{2-|V|}
 \xrightarrow{h}\Sigma^{1-|V|}S/2.
\]
Since $h\in \pi_1(S/2)=\Z/2$, $2h=0$. Therefore $2(1-q_3)$ is null, even before smashing with $S/4$.\qedhere
\end{enumerate}
\end{proof}

\FloatBarrier

\subsection{Nonvanishing and chains with $C_p$-action}\label{sec:chains}\label{subsec:algebraic}

In this subsection, we prove Proposition~\ref{lem:chains}.

\Needspace{5\baselineskip}

\begin{proof}[Proof of Proposition~\ref{lem:chains}]
In each case, $C$ is a bounded complex of free modules, so $\tau$ is zero in the derived category exactly when it is chain null-homotopic.
\begin{enumerate}
\item Suppose $p$ is odd, and let $C$ and $\tau$ be as in Proposition~\ref{lem:chains}(1).
We use $\F_p$-chains with $C_p$-action to represent $H\F_p\wedge X_p$ in $\D(\F_p[C_p])$ \cite[Theorem~5.1.6]{SS}. The free-cell sequence in Lemma~\ref{lem:free-cell-power} identifies the chains of $\Cof(a_\lambda)$ with $\Sigma C$. The two-cell complex of $S/p$ has zero differential over $\F_p$ and trivial $C_p$-action. Tensoring therefore represents $H\F_p\wedge X_p$ by $\Sigma C\oplus\Sigma^2C$. Since $q_{p+1}$ acts only on the circle factor, Lemma~\ref{lem:free-cell-power} represents $H\F_p\wedge f_p$ by $\Sigma\tau\oplus\Sigma^2\tau$.

\[
 \begin{tikzcd}[column sep=large,row sep=large]
 0 \arrow[r] &
 \mathbb F_p[C_p] \arrow[r,"1-g"] \arrow[d,"\tau_1"'] &
 \mathbb F_p[C_p] \arrow[r] \arrow[d,"\tau_0=0"] \arrow[dl,dashed,"h" description] &
 0 \\
 0 \arrow[r] &
 \mathbb F_p[C_p] \arrow[r,"1-g"'] &
 \mathbb F_p[C_p] \arrow[r] & 0.
 \end{tikzcd}
\]

A chain nullhomotopy of $\tau$ would have a single component $h:C_0\longrightarrow C_1$, given by multiplication by an element $h\in\F_p[C_p]$ satisfying
\[
 (1-g)h=0,\qquad h(1-g)=-(1+g+\cdots+g^{p-1}).
\]
The left sides agree by commutativity, but $1+g+\cdots+g^{p-1}\ne0$. Therefore $\tau$ is nonzero in $\D(\F_p[C_p])$, and $f_p$ remains nonzero in the Borel category.

\item Let $p=2$, and let $C$ and $\tau$ be as in Proposition~\ref{lem:chains}(2).
We use $\Z/4$-chains with $C_2$-action to represent $H\Z/4\wedge X_2$ in $\D(\Z/4[C_2])$. The chains of $\Cof(a_{2\sigma})$ are $\Sigma C$. The two-cell complex of $S/4$ has zero differential over $\Z/4$ and trivial $C_2$-action, so tensoring represents $H\Z/4\wedge X_2$ by $\Sigma C\oplus\Sigma^2C$. The map $2(1-q_3)$ acts only on the circle factor, so Lemma~\ref{lem:free-cell-power} represents $H\Z/4\wedge f_2$ by $\Sigma\tau\oplus\Sigma^2\tau$.

\[
 \begin{tikzcd}[column sep=large,row sep=large]
 0 \arrow[r] &
 \mathbb Z/4[C_2] \arrow[r,"1-g"] \arrow[d,"\tau_1"'] &
 \mathbb Z/4[C_2] \arrow[r] \arrow[d,"\tau_0=0"] \arrow[dl,dashed,"h" description] &
 0 \\
 0 \arrow[r] &
 \mathbb Z/4[C_2] \arrow[r,"1-g"'] &
 \mathbb Z/4[C_2] \arrow[r] & 0.
 \end{tikzcd}
\]

A chain nullhomotopy would require an element $h\in\Z/4[C_2]$ satisfying
\[
 (1-g)h=0,\qquad h(1-g)=-2(1+g).
\]
The left sides agree, but $-2(1+g)\ne0$ in $\Z/4[C_2]$. Therefore $\tau$ is nonzero in $\D(\Z/4[C_2])$, and $f_2$ remains nonzero in the Borel category.
\end{enumerate}
In both cases, the underlying map is zero by Proposition~\ref{prop:twisted-nullity} with $H=e$, so the induced map on ordinary homology is zero.
\end{proof}

\FloatBarrier

\section{Induction to all finite groups}\label{sec:all-groups}

In this section, we prove the induction principle. As a corollary, we finish the proof of Theorem~\ref{thm:main} for all nontrivial finite groups.

\begin{theorem}\label{lem:induction}
Let $G$ be a finite group, let $H\leq G$, and let $f:X\to Y$ be a map of $H$-spectra. The induction functor $\Ind_H^G(-)$ has the following properties.
\begin{enumerate}
\item If
\[
 \bigl(S^{-V}\wedge\Res_K^H f\bigr)^K=0
 \qquad(K\leq H,\ V\in\RO(K)),
\]
then $\Ind_H^G f$ has the same vanishing property at every subgroup of $G$.
\item {Induction preserves ghosts for each of the three detector families
in Conjecture~\ref{conj:equivariant}.}
\item If $f\ne0$, then $\Ind_H^G f\ne0$.
\item If $X=Y$, then for every $n\geq1$,
\[
 (\Ind_H^G f)^n=0\quad\Longleftrightarrow\quad f^n=0.
\]
\item The maps $f$ and $\Ind_H^G f$ have the same additive order.
\end{enumerate}
\end{theorem}
\begin{proof}
Let $L\leq G$ and $W\in\RO(L)$. For each representative $g$ of $L\backslash G/H$, let $K_g=H\cap g^{-1}Lg$, and let $W_g$ denote the restriction of $W$ along $K_g\to L$, $k\mapsto gkg^{-1}$. The double-coset decomposition, the projection formula, and the identification of induction with coinduction give an equivalence
\[
 \bigl(S^{-W}\wedge\Res_L^G\Ind_H^G X\bigr)^L
 \simeq
 \bigvee_{[g]\in L\backslash G/H}
 \bigl(S^{-W_g}\wedge\Res_{K_g}^H X\bigr)^{K_g},
\]
natural in $X$. Each summand of the map induced by $f$ is null by hypothesis. This proves~\textup{(1)}. {Taking $\pi_0$ proves~\textup{(2)} for the all-subgroup representation-orbit family. Restricting $W$ to integer degrees proves the integer-graded case. For the $\RO(G)$-graded case, take $W=\Res_L^G V$ with $V\in\RO(G)$. Each $W_g$ is then isomorphic to the restriction of $\Res_H^G V$, so the $\RO(H)$-graded ghost condition on $f$ applies.}

Restricting back to $H$, the identity double coset gives $X$ as a direct summand of $\Res_H^G\Ind_H^G X$ and $f$ as the corresponding block of $\Res_H^G\Ind_H^G f$. Therefore induction preserves nonzero maps. Since induction is additive and preserves composition, applying this argument to integer multiples and composition powers proves~\textup{(3)--(5)}. The same splitting holds in the Borel category.
\end{proof}

\begin{proof}[Proof of Theorem~\ref{thm:main}]
Let $p$ be a prime divisor of $|G|$, and choose a subgroup $K\leq G$ of order $p$. Identify $K$ with $C_p$, and let
\[
 X=\Ind_K^G X_p,\qquad f=\Ind_K^G f_p.
\]
By Propositions~\ref{prop:twisted-nullity} and~\ref{lem:chains}, all representation-twisted fixed-point maps of $f_p$ are null at every subgroup of $C_p$, and $f_p$ remains nonzero in the Borel category. Theorem~\ref{lem:induction} gives the same properties for $f$. The spectrum $X$ is finite and already $p$-local, since $X_p$ is. Induction carries each free $K$-orbit cell to one free $G$-orbit cell, so $X$ has four free orbit cells. The prime-order case in Subsection~\ref{subsec:prime-proof} and Theorem~\ref{lem:induction}(5) give $f$ exact additive order $p$.

Taking $\pi_0$ of the representation-twisted fixed-point maps of $f$ gives
\[
 [S^W,\Res_H^G f]^H=0
 \qquad(H\leq G,\ W\in\RO(H)).
\]
Since $f\ne0$ and $X$ is already $p$-local, this also proves Theorem~\ref{thm:main-results}\textup{(\ref{item:main-failure})}, integrally and $p$-locally.
\end{proof}

\begin{proposition}\label{prop:small-refinements}
For every nontrivial finite group $G$ and every prime $p\mid |G|$, the map $f:X\to X$ constructed in the proof of Theorem~\ref{thm:main} satisfies
\[
 f\circ f=0.
\]
\end{proposition}
\begin{proof}
We first prove the assertion for $f_p$. Let $L$ denote $S/p$ for odd $p$ and $S/4$ for $p=2$. Smashing~\eqref{eq:circle-triangle} with $L$ gives a cofiber sequence
\[
 (C_p)_+\wedge\Sigma L\xrightarrow{j_p}X_p
 \xrightarrow{\partial_p}(C_p)_+\wedge\Sigma^2L.
\]
We have $f_pj_p=0$ by~\eqref{eq:free-power-relations}, and $\Res_e^{C_p}f_p=0$ by Proposition~\ref{prop:twisted-nullity}. Let $F_p=(C_p)_+\wedge\Sigma^2L$. Exactness gives a factorization $f_p=b_p\partial_p$ for some $b_p:F_p\to X_p$. By induction--restriction adjunction, $f_pb_p=0$, since $\Res_e^{C_p}f_p=0$. Therefore $f_p^2=0$.

The claim for $f=\Ind_K^G f_p$ now follows from Theorem~\ref{lem:induction}(4).
\end{proof}

\section{A natural ghost operation}\label{sec:two-cell-ghosts}

In this section, we prove Proposition~\ref{prop:natural-main} for the group $C_p$. The induction principle of Theorem~\ref{lem:induction} then promotes this to arbitrary finite groups. We then illustrate it with examples of $C_p$-equivariant complexes consisting of only two fixed cells.

\subsection{A natural cyclic ghost}

Recall that we denote $C_p=\langle g\rangle$, and recall that the restriction along $i:e\hookrightarrow C_p$ and geometric fixed points induce ring homomorphisms $i^*,\Phi:A(C_p)\longrightarrow\mathbb Z$, respectively. They together give an injection
\begin{equation}\label{eq:burnside-detection}
 (i^*,\Phi):A(C_p)\hookrightarrow\mathbb Z^2,
 \qquad x+y[C_p/e]\longmapsto(x+py,x).
\end{equation}
which coincides with the marks homomorphism.

We first note that the map $1-g$ already vanishes on $C_p$-fixed points when $p$ is odd.

\begin{lemma}\label{lem:odd-cyclic-fixed-null}
For any odd prime $p$ and any $C_p$-equivariant spectrum $X$ and $V\in\RO(C_p)$, the map
\[
 \bigl(S^{-V}\wedge(1-g)\bigr)^{C_p}:
 (S^{-V}\wedge X)^{C_p}\longrightarrow(S^{-V}\wedge X)^{C_p}
\]
is null.
\end{lemma}
\begin{proof}
We first show that $1-g:S^V\to S^V$ is null. Under the identification $[S^V,S^V]^{C_p}\cong A(C_p)$, we regard $1-g$ as an element of the Burnside ring. Since $g^p=\Id$ and $p$ is odd, this forces the underlying mark of $1-g$ to be zero. At the $C_p$-level, $g$ acts identically on the fixed subspace of an actual representation, and this remains true for virtual representations since geometric fixed points preserve smash products and duals of representation spheres \cite[V, Corollary~4.6 and Proposition~4.7]{MM}. Thus, $\Phi(1-g)=0$. Injectivity of \eqref{eq:burnside-detection} gives $1-g=0:S^V\to S^V$.

Let $Y$ be a spectrum with trivial $C_p$-action.
The diagonal action gives
\[
 1-g=0:S^V\wedge Y\longrightarrow S^V\wedge Y.
\]
For every $C_p$-equivariant map $a:S^V\wedge Y\to X$, we have the following commutative diagram by naturality
\[
\begin{tikzcd}[column sep=large,row sep=large]
 S^V\wedge Y \arrow[r,"a"] \arrow[d,"1-g=0"'] &
 X \arrow[d,"1-g"] \\
 S^V\wedge Y \arrow[r,"a"'] & X.
\end{tikzcd}
\]
So $(1-g)\circ a=0$. On the other hand, in the adjunction \cite[V, Proposition 3.4]{MM},
\[
 [S^V\wedge Y,X]^{C_p}
 \cong [Y,(S^{-V}\wedge X)^{C_p}],
\]
$(1-g)\circ a$ is the adjoint of
\begin{equation}\label{eq:adjointY}
     Y\xrightarrow{\widetilde a}(S^{-V}\wedge X)^{C_p}
 \xrightarrow{(S^{-V}\wedge (1-g))^{C_p}}(S^{-V}\wedge X)^{C_p},
\end{equation}
for some \(\widetilde a\) adjoint to \(a\). The composite \eqref{eq:adjointY} is therefore zero for every \(Y\) and \(\widetilde a\). Taking \(Y=(S^{-V}\wedge X)^{C_p}\) and \(\widetilde a=\mathrm{id}\) proves the assertion.
\end{proof}

In general, $1-g:X\to X$ is not necessarily a ghost since its underlying map need not be zero. On the other hand, we will give an example in Section~\ref{subsec:two-cell-odd} where $1-g$ becomes an actual, nontrivial ghost.

To remedy the problem at the underlying level and to incorporate the $p=2$ case, we consider the Burnside ring $A(C_p)$, and note that $p-[C_p/e]$ is nonzero though its underlying mark is trivial. Then, multiplying the cyclic operation $1-g$ by $p-[C_p/e]$ gives the following natural endomorphism
\[
 \Psi_X=(p-[C_p/e])(1-g):X\longrightarrow X.
\]

\begin{proposition}[Proposition~\ref{prop:natural-main}]\label{prop:natural-ghost}
For any prime $p$ and any $C_p$-equivariant spectrum $X$, the map $\Psi_X$ satisfies
\[
 \bigl(S^{-V}\wedge\Res_H^{C_p}\Psi_X\bigr)^H=0 \qquad(H\leq C_p,\ V\in\RO(H))
\]
as a map of classical spectra. In particular, $\Psi_X$ is a representation-orbit ghost.
\end{proposition}
\begin{proof}
For $H=e$, the claim follows from $i^*(p-[C_p/e])=0$.

For $H=C_p$ and any $V\in RO(C_p)$, we first show that 
\[\Psi_{S^V}=(p-[C_p/e])(1-g):S^V\to S^V\]
is null. Under the identification $[S^V,S^V]^{C_p}\cong A(C_p)$, we regard $\Psi_{S^V}$ as an element of the Burnside ring. By the case $H=e$, it remains to check its image under $\Phi$. The same argument in Lemma~\ref{lem:odd-cyclic-fixed-null} gives $\Phi(g)=1$; therefore,
\[
 \Phi(\Psi_{S^V})=p(1-\Phi(g))=0.
\]
Injectivity of \eqref{eq:burnside-detection} gives $\Psi_{S^V}=0$.

Using a similar argument in the second part of the proof of Lemma~\ref{lem:odd-cyclic-fixed-null}, we can extend the vanishing to $\Psi_X$.
\end{proof}

\begin{corollary}\label{cor:induced-ghost}
Let \(G\) be a finite group containing \(C_p\). For every \(C_p\)-spectrum \(X\), let \(\Psi_X\) be the endomorphism of Proposition~\ref{prop:natural-ghost}. Then the induced endomorphism \(\Ind_{C_p}^G \Psi_X=G_+\wedge_{C_p}\Psi_X\) satisfies
\[
 \bigl(S^{-V}\wedge\operatorname{Res}_H^G\Ind_{C_p}^G \Psi_X\bigr)^H=0
 \qquad(H\leq G,\ V\in RO(H))
\]
as a map of classical spectra. In particular, \(\Ind_{C_p}^G \Psi_X\) is a ghost.
\end{corollary}

\begin{proof}
Apply Theorem~\ref{lem:induction} to the vanishing in Proposition~\ref{prop:natural-ghost}.
\end{proof}

\subsection{A two-cell ghost at odd primes}\label{subsec:two-cell-odd}
Fixing an odd prime $p$, consider the cofiber sequence
\begin{equation}\label{eq:two-cell-burnside-cofiber}
 S^0\xrightarrow{[C_p/e]}S^0\xrightarrow{\iota}\Cof([C_p/e])
   \xrightarrow{\delta}S^1.
\end{equation}
On $\Cof([C_p/e])$, $1-g$ is actually a ghost.

\begin{proposition}\label{prop:two-cell-odd-ghost}
The map
\[ 1-g:\Cof([C_p/e])\longrightarrow\Cof([C_p/e])\]
has zero underlying map. In particular, it satisfies
\[
 \bigl(S^{-V}\wedge\Res_H^{C_p}(1-g)\bigr)^H=0
 \qquad(H\leq C_p,\ V\in\RO(H)).
\]
and is a representation-orbit ghost.
\end{proposition}
\begin{proof}
The cofiber sequence in \eqref{eq:two-cell-burnside-cofiber} becomes
\begin{equation}\label{eq:modpcofseq}
    S^0\xrightarrow{p}S^0\overset{\iota}{\to} S/p\to S^1
\end{equation} when restricted to the underlying. In the long exact sequence induced from \eqref{eq:modpcofseq},
\[\pi_1(S/p)\to [S/p, S/p] \overset{\iota^*}{\to} \pi_0(S/p)\to \pi_0(S/p)\]
since $\pi_1(S/p)=0$ by Lemma~\ref{lem:moore}, $\iota^*$ is an injection. Since $1-g$ is null on $S^0$, by naturality,
\[\begin{tikzcd}
S^0 \arrow[r, "\iota"] \arrow[d, "i^*(1-g)"'] & S/p \arrow[d, "i^*(1-g)"] \\
S^0 \arrow[r, "\iota"]                   & S/p                 
\end{tikzcd}\]
we have $i^*(1-g)\circ \iota: S^0 \to S/p$ is null. By injectivity, we have $i^*(1-g)$ is null. Combining with Lemma~\ref{lem:odd-cyclic-fixed-null}, the claim follows.
\end{proof}

We first describe a factorization for a morphism of
cofiber sequences in a stable $\infty$-category $\mathcal C$:
\[
\begin{tikzcd}[column sep=large,row sep=large]
 A \arrow[r,"u"] \arrow[d,"a"'] &
 B \arrow[r,"\iota"] \arrow[d,"b"] &
 C \arrow[r,"\delta"] \arrow[d,"c"] & \Sigma A \arrow[d,"\Sigma a"] \\
 A' \arrow[r,"u'"'] & B' \arrow[r,"\iota'"'] &
 C' \arrow[r,"\delta'"'] & \Sigma A'.
\end{tikzcd}
\]
Choose nullhomotopies of $a$ and $b$ as $2$-simplices
$A\to0\to A'$ and $B\to0\to B'$, with long edges $a$ and $b$.
The given morphism and these nullhomotopies admit an extension
to a $3$-simplex in $\Fun(\Delta^1,\mathcal C)$ with spine
\[
 (A\xrightarrow{u}B)\xrightarrow{(\Id,0)}(A\to0)
 \longrightarrow(0\to B')\xrightarrow{(0,\Id)}(A'\xrightarrow{u'}B'),
\]
and the given morphism as its long edge. Applying the cofiber functor
\cite[Remark~1.1.1.7]{HA} gives a $3$-simplex with spine
\begin{equation}\label{eq:two-cell-coherent-factorization}
 C\xrightarrow{\delta}\Sigma A\xrightarrow{\theta}B'
 \xrightarrow{\iota'}C'
\end{equation}
and long edge $c$, exhibiting $[c]=[\iota'\theta\delta]$ in
$\Ho(\mathcal C)$.

To apply this construction, let us specify the nullhomotopy of $1-g$ using the spectral Mackey functor perspective of \cite{GuillouMay}. Let $\operatorname{Span}_{C_p}$ denote the $\infty$-category of finite $C_p$-sets and spans. We have
\begin{equation}\label{eq:guilloumay}
    \Sp^{C_p}\simeq \Fun^\oplus(\operatorname{Span}_{C_p}^{\mathrm{op}},\Sp),
\end{equation}
where $\Fun^\oplus$ denotes the functors that preserve finite direct sums. Under this equivalence, the spectral Mackey functor represented by a finite $C_p$-set $T$ corresponds to $\Sigma^\infty_{C_p}T_+$ \cite[Section~2.5]{GuillouMay}. In particular, the point $*$ corresponds to $S^0$, and the endomorphism of $* \to *$ represented by the span
\begin{equation}\label{eq:spanCp/e}
*\longleftarrow C_p/e\longrightarrow *
\end{equation}
corresponds to $[C_p/e]:S^0\to S^0$.

Let $\Inf_e^{C_p}:\Sp\to\Sp^{C_p}$ denote inflation along
$C_p\to e$, which is left adjoint to the categorical fixed point:
\[
 \operatorname{Map}_{\Sp^{C_p}}(\Inf_e^{C_p}Y,X)
 \simeq\operatorname{Map}_{\Sp}(Y,X^{C_p}).
\]
Under the identification of \eqref{eq:guilloumay}, categorical fixed points are evaluation at $*$ \cite[Construction~3.5]{GuillouMay}. Since translation by $g$ on $*$ is the identity, it gives a specified natural nullhomotopy of $(1-g)^{C_p}$. Under the adjunction above, the map $1-g:\Inf_e^{C_p}Y\to\Inf_e^{C_p}Y$ corresponds to
\[
 Y\xrightarrow{\eta}(\Inf_e^{C_p}Y)^{C_p}
 \xrightarrow{(1-g)^{C_p}}(\Inf_e^{C_p}Y)^{C_p},
\]
where $\eta$ denotes the unit of the adjunction. This gives a natural nullhomotopy
\begin{equation}\label{eq:two-cell-inflation-nullhomotopy}
 (1-g)\Inf_e^{C_p}\simeq0:
 \Inf_e^{C_p}\Longrightarrow\Inf_e^{C_p},
\end{equation}
and we use this choice throughout.

Under the equivalence in \eqref{eq:guilloumay},
the natural automorphism of $\operatorname{Span}_{C_p}$ given by translation by $g$ corresponds to the usual action of $g$ on genuine $C_p$-spectra. For the span \eqref{eq:spanCp/e}, the endpoint translations are identities, while the naturality $2$-morphism is $g:C_p/e\to C_p/e$. Subtracting this naturality square from the identity and taking cofibers gives
\begin{equation}\label{eq:two-cell-coherent-diagram}
\begin{tikzcd}[column sep=large,row sep=large]
 S^0 \arrow[r,"{[C_p/e]}"] \arrow[d,"1-g \simeq 0"'] &
 S^0 \arrow[r,"\iota"] \arrow[d,"1-g \simeq 0"'] &
 \Cof([C_p/e]) \arrow[r,"\delta"] \arrow[d,"1-g"'] &
 S^1 \arrow[d,"1-g \simeq 0"] \\
 S^0 \arrow[r,"{[C_p/e]}"'] &
 S^0 \arrow[r,"\iota"'] &
 \Cof([C_p/e]) \arrow[r,"\delta"'] & S^1.
\end{tikzcd}
\end{equation}
The indicated nullhomotopies are the values of
\eqref{eq:two-cell-inflation-nullhomotopy}. 

\begin{proposition}\label{prop:two-cell-factorization}
For any odd prime $p$, there is a factorization
\begin{equation}\label{eq:two-cell-factorization}
\begin{tikzcd}
{\Cof([C_p/e])} \arrow[d, "1-g"'] \arrow[r, "\delta"] & S^1 \arrow[d, dashed, "\gamma"] \\
{\Cof([C_p/e])}                                       & S^0. \arrow[l, "\iota"] 
\end{tikzcd}
\end{equation}
\end{proposition}
\begin{proof}
With the choices in \eqref{eq:two-cell-coherent-diagram} specified above, \eqref{eq:two-cell-coherent-factorization} implies that $1-g$ on $\Cof([C_p/e])$ admits the claimed factorization, where we denote the middle map by $\gamma$.
\end{proof}

\begin{figure}[!htbp]
\centering
\begin{tikzpicture}[x=1cm,y=1cm,scale=.94,transform shape,font=\small,
 dimcell/.style={circle,draw,fill=white,minimum size=.64cm,inner sep=.6pt,font=\tiny}]
 \foreach \x/\n in {0/s,5.5/t} {
  \node[dimcell] (\n0) at (\x,0) {$0$};
  \node[dimcell] (\n1) at (\x,1.65) {$1$};
  \draw[attachment] (\n1) -- node[right=3pt] {$[C_p/e]$} (\n0);
 }
 \draw[cellmap] (s1.east) -- node[midway,above=3pt] {$\gamma$} (t0.west);
 \node[anchor=east] at (-1.5,0.825) {$1-g:$};
\end{tikzpicture}
\caption{A cell diagram for $1-g$ on $\Cof([C_p/e])$ (odd $p$)}
\label{fig:two-cell-odd}
\end{figure}

Under group completion of the span groupoid, the map $\gamma$ is the loop of $g$ in $B\operatorname{Aut}_{C_p}(C_p/e)\simeq BC_p$, translated to the zero component.

In particular, we obtain a global relation as follows.
\begin{corollary}\label{cor:two-cell-odd-relation}
For odd prime $p$ and any $x\in\pi_V^{C_p}S^0$,
\[
 [C_p/e]\cdot x=0\quad \Longrightarrow \quad
 \gamma x\in[C_p/e]\cdot \pi_{V+1}^{C_p}S^0.
\]
\end{corollary}
\begin{proof}
Since $[C_p/e]\cdot x=0$, there exists a lift
$z:S^{V+1}\to\Cof([C_p/e])$ with $\delta z=\Sigma x$.
Since $1-g$ is a representation-orbit ghost,
\[
 0=(1-g) \circ z=\iota \gamma\delta z=\iota\circ  \gamma\Sigma x.
\]
and this implies
$\gamma x\in[C_p/e]\cdot \pi_{V+1}^{C_p}S^0$.
\end{proof}

The factorization in Proposition~\ref{prop:two-cell-factorization} allows us to show that $1-g$ is detected by the homology theory $\Inf_e^{C_p} H\Z$.

\begin{theorem}\label{thm:two-cell-odd-nonzero}
For an odd prime $p$, the induced map
\[
\Inf_e^{C_p}H\Z\wedge(1-g): \Inf_e^{C_p}H\Z\wedge\Cof([C_p/e]) \longrightarrow \Inf_e^{C_p}H\Z\wedge\Cof([C_p/e])
\]
is nonzero. Thus $1-g$ is a nontrivial representation-orbit ghost.
\end{theorem}
\begin{proof}
Consider the equivalence \cite[Corollary~4.11]{PSW}
\begin{equation}\label{eq:two-cell-linear-mackey}
\operatorname{Mod}_{\Inf_e^{C_p}H\Z}\simeq \Fun^\oplus\bigl(\operatorname{Span}_{C_p}^{\mathrm{op}},\mathcal D(\Z)\bigr).
\end{equation}
Using the $\mathcal D(\Z)$-linear structure in \eqref{eq:two-cell-linear-mackey}, we can define an algebra $R$ in $\mathcal D(\Z)$ by
\[
R:=\operatorname{End}_{\operatorname{Mod}_{\Inf_e^{C_p}H\Z}}(\Inf_e^{C_p}H\Z)\simeq(\Inf_e^{C_p}H\Z)^{C_p}.
\]
Write $\mathcal D(R)$ for the $\infty$-category of right $R$-modules in $\mathcal D(\Z)$. Taking categorical fixed points gives an exact functor
\begin{equation}\label{eq:infHZfunctor}
\begin{aligned}
\operatorname{Mod}_{\Inf_e^{C_p}H\Z}&\longrightarrow\mathcal D(R),\\
X&\longmapsto X^{C_p},
\end{aligned}
\end{equation}
where the $R$-action is induced by the $\Inf_e^{C_p}H\Z$-module structure on $X$. We will show that the image of the map in the theorem is nonzero in $\mathcal D(R)$.

Denote by $t:R\to R$ the morphism representing $[C_p/e]$. Then, the functor \eqref{eq:infHZfunctor} sends $\Inf_e^{C_p}H\Z\wedge\Cof([C_p/e])$ to $\Cof(t)$, and by Proposition~\ref{prop:two-cell-factorization}, it sends $\Inf_e^{C_p}H\Z\wedge(1-g)$ to
\[
\Cof(t)\xrightarrow{\delta}\Sigma R\xrightarrow{\gamma}R\xrightarrow{\iota}\Cof(t).
\]

Since $\Inf_e^{C_p}$ and $(\_)^{C_p}$ are both colimit-preserving, 
\[(\Inf_e^{G}X)^{G} \simeq X\wedge (S^0)^G.\]
Then, the tom Dieck splitting \cite[Theorem~V.11.1]{LMS} gives
\[
(\Inf_e^{C_p}H\Z)^{C_p}\simeq H\Z\wedge(S^0)^{C_p}\simeq H\Z\vee\bigl(H\Z\wedge\Sigma^\infty_+BC_p\bigr).
\]
Therefore, the identification $H_n R \cong \pi_n^{C_p}(\Inf_e^{C_p} H\Z)$ gives
\begin{equation}\label{eq:two-cell-end-homology}
H_0R=A(C_p),\qquad H_1R=H_1(BC_p;\Z)=\Z/p.
\end{equation}
By the construction of $\gamma$, its image
$\pm[g]$ generates $H_1R=\Z/p$.

Now, consider the cofiber sequence
\begin{equation}\label{eq:cofoft}
R\xrightarrow{t}R\xrightarrow{\iota}\Cof(t)\xrightarrow{\delta} \Sigma R.
\end{equation}
As $t$ represents the Burnside element $[C_p/e]$, its action on equivariant homotopy groups is the composite of restriction and transfer
\[
\pi_1^{C_p}(\Inf_e^{C_p}H\Z)\xrightarrow{\operatorname{res}_e^{C_p}}
\pi_1(H\Z)\xrightarrow{\operatorname{tr}_e^{C_p}}
\pi_1^{C_p}(\Inf_e^{C_p}H\Z).
\]
Since $\pi_1(H\Z)=0$, this composite vanishes. Thus, the map induced by $t$ on $H_1R$ is trivial. The long exact sequence induced by \eqref{eq:cofoft} therefore gives
\[
H_1R\xrightarrow{0}H_1R\xrightarrow{\iota_*}H_1\Cof(t).
\]
It follows that $\iota_*$ is injective. Since $\gamma$ is nonzero in $H_1R$, we have $\iota\gamma\ne0\in H_1\Cof(t)$.

Finally, applying $[-,\Cof(t)]_{\mathcal D(R)}$ to \eqref{eq:cofoft} gives
\[
H_1\Cof(t)\xrightarrow{t}H_1\Cof(t)\xrightarrow{\delta^*}[\Cof(t),\Cof(t)]_{\mathcal D(R)}.
\]
By a similar argument as above, $t$ is the composite of restriction and transfer which factors through $\pi_1(H\Z/p)=0$. Thus, $t$ is trivial and $\delta^*$ is injective. Hence the induced map $\iota \gamma\delta=\delta^*(\iota\gamma)\ne0$.
\end{proof}

\subsection{The natural ghost at prime 2}\label{subsec:two-cell-two}

We now take $p=2$ and let $g$ be the nontrivial element of $C_2$. Write $\eta\in\pi_1S^0$ for the classical Hopf class. The construction of $\gamma$ in Proposition~\ref{prop:two-cell-factorization} also applies to $C_2$. After forgetting the action, $g:C_2/e\to C_2/e$ is a transposition, whose loop represents $\eta$ under the Barratt--Priddy--Quillen equivalence \cite[Theorem~2.12]{GuillouMay}. Thus $\Res_e^{C_2}(\gamma)=\eta$. Applying \eqref{eq:two-cell-coherent-factorization} to \eqref{eq:two-cell-coherent-diagram} and then restricting to the underlying gives
\[
 \Res_e^{C_2}(1-g)=\iota\eta\delta=2:S/2\longrightarrow S/2.
\]
This map acts nontrivially on
$\pi_2(S/2)\cong\Z/4$, so $1-g$ on $\Cof([C_2/e])$ is no longer
a ghost. We therefore use
the natural ghost $\Psi$ instead.

We consider the following general construction. For $k>0$, choose $x,y\in\pi_kS^0$, and regard elements in the nonequivariant stable stems as equivariant maps by $\Inf_e^{C_2}$. Consider the cofiber sequence
\begin{equation}\label{eq:two-cell-two-cofiber}
S^k\xrightarrow{x+y[C_2/e]}S^0\xrightarrow{\iota} \Cof(x+y[C_2/e])\xrightarrow{\delta}S^{k+1}.
\end{equation}
By Proposition~\ref{prop:natural-ghost}, the endomorphism
\[
 \Psi=(2-[C_2/e])(1-g): \Cof(x+y[C_2/e])\longrightarrow\Cof(x+y[C_2/e])
\]
is a representation-orbit ghost. It admits a top-to-bottom factorization similar to Proposition~\ref{prop:two-cell-factorization}.

\begin{proposition}\label{prop:two-cell-two-factorization}
For the cofiber sequence \eqref{eq:two-cell-two-cofiber}, one has
\begin{equation}\label{eq:two-cell-two-factorization}
 \Psi=\iota\circ\eta y[C_2/e]\circ\delta:
 \Cof(x+y[C_2/e])\longrightarrow\Cof(x+y[C_2/e]),
\end{equation}
where $\eta y[C_2/e]:S^{k+1}\to S^0$ is the indicated product in the equivariant stable homotopy ring.
\end{proposition}

\begin{figure}[!htbp]
\centering
\begin{tikzpicture}[x=1cm,y=1cm,scale=.94,transform shape,font=\small,
 dimcell/.style={circle,draw,fill=white,minimum size=.64cm,inner sep=.6pt,font=\tiny}]
 \foreach \x/\n/\side in {0/s/left,5.5/t/right} {
  \node[dimcell] (\n0) at (\x,0) {$0$};
  \node[dimcell] (\n1) at (\x,1.65) {$k+1$};
  \draw[attachment] (\n1) -- node[\side=3pt] {$x+y[C_2/e]$} (\n0);
 }
 \draw[cellmap] (s1.east) -- node[midway,above=3pt] {$\eta y[C_2/e]$} (t0.west);
 \node[anchor=east] at (-2.2,.825) {$\Psi:$};
\end{tikzpicture}
\caption{A cell diagram for $\Psi$ on $\Cof(x+y[C_2/e])$.}
\label{fig:two-cell-two}
\end{figure}
\FloatBarrier

\begin{proof}
In the following diagram, choose the nullhomotopies on the
spheres from \eqref{eq:two-cell-inflation-nullhomotopy},
with $p=2$:
\begin{equation}\label{eq:two-cell-two-coherent-diagram}
\begin{tikzcd}[column sep=large,row sep=large]
 S^k \arrow[r,"{x+y[C_2/e]}"] \arrow[d,"1-g\simeq0"'] &
 S^0 \arrow[r,"\iota"] \arrow[d,"1-g\simeq0"'] &
 \Cof(x+y[C_2/e]) \arrow[r,"\delta"] \arrow[d,"1-g"'] &
 S^{k+1} \arrow[d,"1-g\simeq0"] \\
 S^k \arrow[r,"{x+y[C_2/e]}"'] &
 S^0 \arrow[r,"\iota"'] &
 \Cof(x+y[C_2/e]) \arrow[r,"\delta"'] & S^{k+1}.
\end{tikzcd}
\end{equation}
For the inflated map $x$, naturality of \eqref{eq:two-cell-inflation-nullhomotopy} implies that the middle map in \eqref{eq:two-cell-coherent-factorization} is zero. For the composite
\[
 y[C_2/e]:S^k\xrightarrow{[C_2/e]}S^k\xrightarrow{y}S^0,
\]
pasting the naturality squares gives the middle map $y\gamma:S^{k+1}\to S^0$. Therefore, $1-g=\iota\,y\gamma\,\delta$.

Since $g^2=1$, its span loop satisfies $2\gamma=0$. Since $\Res_e^{C_2}(\gamma)=\eta$, the projection formula implies $[C_2/e]\gamma=\operatorname{tr}_e^{C_2}(\eta)=\eta[C_2/e]$. The claim follows by multiplying the factorization by $2-[C_2/e]$.
\end{proof}

\begin{corollary}\label{cor:two-cell-two-divisibility}
For any $z\in\pi_V^{C_2}S^0$,
\[
 (x+y[C_2/e])\cdot z=0 \quad\Longrightarrow\quad \eta y[C_2/e]\cdot z\in(x+y[C_2/e])\cdot \pi_{V+1}^{C_2}S^0.
\]
\end{corollary}
\begin{proof}
Since $(x+y[C_2/e])\cdot z=0$, there exists a lift $\widetilde z:S^{V+k+1}\to\Cof(x+y[C_2/e])$ with $\delta\widetilde z=\Sigma^{k+1}z$. Since $\Psi$ is a ghost, the factorization \eqref{eq:two-cell-two-factorization} gives 
\[\iota \eta y[C_2/e]z=0.\] 
This implies $\eta y[C_2/e]\cdot z\in(x+y[C_2/e])\cdot \pi_{V+1}^{C_2}S^0$.
\end{proof}

\begin{proposition}\label{prop:two-cell-two-nonzero}
For $k>0$ and $x,y\in\pi_kS^0$, the ghost
\[
 \Psi:\Cof(x+y[C_2/e])\longrightarrow\Cof(x+y[C_2/e])
\]
is nonzero if and only if 
\[\eta x\neq 0,\quad  \eta y\neq 0, \quad \mathrm{and}\quad  \eta x\neq \eta y.\]
\end{proposition}
\begin{proof}
The tom Dieck splitting \cite[Theorem~V.11.1]{LMS} gives
\begin{equation}\label{eq:tomdiecksplitting}
    (S^0)^{C_2}\simeq S^0\vee S^0\vee\Sigma^\infty BC_2.
\end{equation}

In particular, $\pi_{-k}^{C_2}S^0=0$ for $k>0$.

Let $f:S^k\to S^0$ be any $C_2$-equivariant map with $k>0$,
and consider the cofiber sequence \[ S^k\xrightarrow{f}S^0\xrightarrow{\iota}\Cof(f) \xrightarrow{\delta}S^{k+1}.\]
For any $c:S^{k+1}\to S^0$, $\iota c\delta=0$ if and only if $\iota c=h\Sigma f$ for some $h:S^1\to\Cof(f)$. Since any map $S^1 \to S^{k+1}$ is null for $k>0$, $h$ lifts to a map $v:S^1\to S^0$. Exactness then gives $c-v\Sigma f=fw$ for some $w:S^{k+1}\to S^k$. Conversely, if $c$ can be written as $v\Sigma f+fw$, then $\iota c\delta=0$. Thus,
\begin{equation}\label{eq:two-cell-two-coefficient-kernel}
 \ker\bigl(\pi_{k+1}^{C_2}S^0       \xrightarrow{\;\textstyle c\mapsto \iota c\delta\;}       [\Cof(f),\Cof(f)]^{C_2}\bigr) =f\cdot \pi_1^{C_2}S^0+\pi_1^{C_2}S^0\cdot f.
\end{equation}

By the tom Dieck splitting \cite[Theorem~V.11.1]{LMS}, 
\[
 \pi_1^{C_2}S^0=\F_2\{\eta,\eta[C_2/e],\gamma\},
\]
where $\eta$ and $\eta[C_2/e]$ come from the two sphere summands and $\gamma$ projects to the generator of $\pi_1\Sigma^\infty BC_2=H_1(BC_2;\Z)=\Z/2$.

Because of the relations $[C_2/e]^2=2[C_2/e]$ and $[C_2/e]\gamma=\eta[C_2/e]$, for $f=x+y[C_2/e]$, the kernel in \eqref{eq:two-cell-two-coefficient-kernel} consists of
\begin{equation}\label{eq:two-cell-two-kernel-generators}
 \F_2\{\eta x+\eta y[C_2/e], \eta x[C_2/e], x\gamma+\eta y[C_2/e]\}.
\end{equation}
If either $\eta y=0$, or $\eta x=0$, or $\eta x=\eta y$, the element $\eta y[C_2/e]$ is either zero or in the kernel. Hence $\Psi=0$.

Conversely, suppose that $\eta x\neq0$, $\eta y\neq0$, and $\eta x\neq \eta y$. Suppose $\eta y[C_2/e]$ is still in the kernel, say 
\[\eta y[C_2/e]=a\cdot (\eta x+\eta y[C_2/e])+b\cdot \eta x[C_2/e]+c\cdot (x\gamma+\eta y[C_2/e]).\]
Applying the geometric fixed point functor, since $\Phi(\gamma)=0$ and $\Phi([C_2/e])=0$, this forces $a=0$. Then, applying $\Res^{C_2}_e$, since $\Res^{C_2}_e(\gamma)=\eta$ and $\eta x\neq 0$, this implies $c=0$. Thus $\eta y[C_2/e]=b\cdot \eta x[C_2/e]$ for some $b\in\F_2$. Note that the second summand in the tom Dieck splitting \eqref{eq:tomdiecksplitting} is multiplication by $[C_2/e]$. Therefore,
\begin{equation}\label{eq:two-cell-free-summand-injection}
 \pi_jS^0\longrightarrow\pi_j^{C_2}S^0,
 \qquad z\longmapsto z[C_2/e],\qquad j\in\Z,
\end{equation}
is injective. We conclude that $\eta y=b \cdot \eta x$, which gives a contradiction. 
\end{proof}

\begin{example}\label{ex:two-cell-two}
Let $x=\eta\sigma$ and $y=\epsilon$ in $\pi_8S^0$, where $\sigma\in\pi_7S^0$ is the stable Hopf class. The classes $\eta^2\sigma$ and $\eta\epsilon$ are distinct nonzero elements in $\pi_9S^0$. Therefore,
\[
 \Psi: \Cof(\eta\sigma+\epsilon[C_2/e]) \longrightarrow\Cof(\eta\sigma+\epsilon[C_2/e])
\]
is a nonzero representation-orbit ghost on a finite $C_2$-spectrum with two fixed orbit cells in degree zero and degree nine.
\end{example}

\section{Arbitrarily long ghosts}\label{sec:long-ghosts}

In this section, we prove that the powers of the natural ghost given by Proposition~\ref{prop:natural-ghost} are nontrivial on $\Sigma^\infty_+(BS^1)^p$. We then pass this construction to finite products of complex projective spaces, producing nontrivial arbitrarily long ghosts on finite $C_p$-spectra. By Theorem~\ref{lem:induction}, this construction can be promoted to every nontrivial finite group which proves Theorem~\ref{thm:long-main}.

\subsection{A nonnilpotent ghost}\label{subsec:nonvanishing}

Having constructed a natural ghost on every \(C_p\)-spectrum, we now ask whether it can be nonzero. The following theorem shows that on \(\Sigma^\infty_+(BS^1)^p\) it is not nilpotent. In fact, no positive power is annihilated by a power of \(p\).

We equip \((S^1)^p\) with the \(C_p\)-action
\[
 g(x_0,\ldots,x_{p-1})=(x_1,\ldots,x_{p-1},x_0),
\]
which induces the same coordinate permutation on \((BS^1)^p\).

\begin{theorem}\label{thm:nonvanishing}
For \(X=\Sigma^\infty_+(BS^1)^p\), the ghost
\[
 \Psi_X=(p-[C_p/e])(1-g):X\longrightarrow X
\]
satisfies
\[
 p^m \Psi_X^n\ne0\quad\text{in}\quad\bigl[\Sigma^\infty_+(BS^1)^p,\Sigma^\infty_+(BS^1)^p\bigr]^{C_p}
 \qquad(m\ge0,\ n\ge1).
\]
\end{theorem}

\begin{proof}
The May--Snaith--Zelewski splitting \cite[Theorem 2.3]{MSZ} requires a finite source group. We therefore restrict along a \(C_p\)-equivariant embedding \(\varphi:C_{p^2}\hookrightarrow(S^1)^p\) and prove that \(p^m \Psi_X^n\varphi\ne0\).

Write \(C_{p^2}=\langle h\rangle\). Choose a primitive \(p^2\)-th root of unity \(\zeta\in S^1\), and define
\[
 \varphi:C_{p^2}\longrightarrow(S^1)^p,
 \qquad
 \varphi(h)=\bigl(\zeta,\zeta^{1+p},\ldots,
                         \zeta^{1+p(p-1)}\bigr).
\]
The congruence \((1+p)^j\equiv1+pj\pmod{p^2}\) shows that \(h\mapsto h^{1+p}\) has order \(p\). This defines a \(C_p\)-action on \(C_{p^2}\) by \(gh=h^{1+p}\). Since \((1+pi)(1+pj)\equiv1+p(i+j)\pmod{p^2}\), we have
\[
 \varphi(g^jh)=\varphi(h^{1+pj})
   =\bigl(\zeta^{1+pj},\zeta^{1+p(j+1)},\ldots,
                            \zeta^{1+p(j+p-1)}\bigr)
   =g^j\varphi(h).
\]
This proves that \(\varphi\) is \(C_p\)-equivariant.

We use a splitting of the mapping spectrum to separate the maps \(g^j\varphi\). For \(L\le C_{p^2}\) and \(\rho:L\to(S^1)^p\), define
\[
 W_\rho=\bigl(C_{p^2}\times(S^1)^p\bigr)/\operatorname{graph}(\rho).
\]
Since \(C_{p^2}\times(S^1)^p\) is abelian, the splitting of May--Snaith--Zelewski \cite[Theorem 2.3]{MSZ} takes the form
\begin{equation}\label{eq:msz}
 \xi:
 \bigvee_{\substack{L\le C_{p^2}\\\rho:L\to(S^1)^p}}
 (\Sigma^\infty_+BW_\rho)^\wedge_p
 \xrightarrow{\simeq} F\left(\Sigma^\infty_+BC_{p^2},\Sigma^\infty_+(BS^1)^p\right)^\wedge_p.
\end{equation}
Here \(F\) is the classical mapping spectrum. The wedge is finite, so \(p\)-completion may be taken summandwise.

The \(C_p\)-action on the source preserves each subgroup \(L\). The map from the \(\rho\)-summand to the \(g\rho g^{-1}\)-summand is induced by
\[
 g:W_\rho\longrightarrow W_{g\rho g^{-1}},
 \qquad [(u,x)]\longmapsto[(gu,gx)],
\]
where \(u\in C_{p^2}\) and \(x\in(S^1)^p\). The target carries the conjugation action \(\psi\mapsto g\psi g^{-1}\). We check that \(\xi\) is \(C_p\)-equivariant by describing its summand maps. Before completion, the map on the \(\rho\)-summand is adjoint to the composite \cite[Theorem 8 and Remarks 9]{LMM}
\begin{equation}\label{eq:correspondence}
 \Sigma^\infty_+B(C_{p^2}\times W_\rho)
 \longrightarrow
 \Sigma^\infty_+B(C_{p^2}\times(S^1)^p)
 \xrightarrow{B\operatorname{pr}_{(S^1)^p}}\Sigma^\infty_+(BS^1)^p.
\end{equation}
To construct the first map in \eqref{eq:correspondence}, consider the homomorphism
\[
 \kappa_\rho:C_{p^2}\times(S^1)^p\longrightarrow C_{p^2}\times W_\rho,
 \qquad (u,x)\longmapsto(u,[(u,x)]).
\]
Since \([(1,x)]=1\) forces \(x=1\), the map \(\kappa_\rho\) is injective. The surjective homomorphism
\[
 C_{p^2}\times W_\rho\longrightarrow C_{p^2}/L,
 \qquad (u,[(u',x)])\longmapsto u'u^{-1}L
\]
has kernel \(\operatorname{im}(\kappa_\rho)\), so \(B\kappa_\rho\) is modeled by a covering of degree \([C_{p^2}:L]\). Then the first arrow in \eqref{eq:correspondence} is the transfer along this covering.

The identities
\[
 \kappa_{g\rho g^{-1}}(gu,gx)=(g\times g)\kappa_\rho(u,x),
 \qquad \operatorname{pr}_{(S^1)^p}(gu,gx)=g\operatorname{pr}_{(S^1)^p}(u,x)
\]
ensure the compatibility of \(\kappa_\rho\) and \(\operatorname{pr}_{(S^1)^p}\) with the \(C_p\)-actions. By the functoriality of the bar construction and the transfer, the family \eqref{eq:correspondence} is coherently \(C_p\)-equivariant. Adjunction and completion therefore make \eqref{eq:msz} an equivalence in \(\Fun(BC_p,\Sp)\).

We now consider the summands with \(L=C_{p^2}\) and \(\rho=g^j\varphi\), \(0\le j<p\). The first coordinates of \(g^j\varphi(h)\) are the distinct roots \(\zeta^{1+pj}\), {so the maps \(g^j\varphi\) correspond to distinct summands.} The \(C_p\)-equivariance of \(\varphi\) gives \(g(g^j\varphi)g^{-1}=g^j\varphi\), so each summand is \(C_p\)-invariant. To compute the action and the adjoint \eqref{eq:correspondence} of the \(g^j\varphi\)-summand, consider the isomorphism
\begin{equation}\label{eq:weyl}
 W_{g^j\varphi}\xrightarrow{\cong}(S^1)^p,
 \qquad [(u,x)]\longmapsto\bigl(g^j\varphi(u)\bigr)^{-1}x
\end{equation}
whose inverse is \(x\mapsto[(1,x)]\). The identity
\[
 \bigl(g^j\varphi(gu)\bigr)^{-1}gx
   =g\left(\bigl(g^j\varphi(u)\bigr)^{-1}x\right)
\]
shows that \eqref{eq:weyl} is \(C_p\)-equivariant for the quotient action on \(W_{g^j\varphi}\) and cyclic permutation on \((S^1)^p\). Under this identification, \(\kappa_{g^j\varphi}\) is the isomorphism
\[
 \kappa_{g^j\varphi}:C_{p^2}\times(S^1)^p
 \xrightarrow{\cong}C_{p^2}\times(S^1)^p,
 \qquad (u,x)\longmapsto
 \bigl(u,(g^j\varphi(u))^{-1}x\bigr).
\]
The transfer in \eqref{eq:correspondence} is therefore induced by its inverse, so the adjoint of this summand is induced by
\begin{equation}\label{eq:adjoint}
 C_{p^2}\times(S^1)^p\longrightarrow(S^1)^p,
 \qquad (u,x)\longmapsto g^j\varphi(u)x.
\end{equation}

The projection \(\operatorname{pr}_j\) onto each of these invariant summands is \(C_p\)-equivariant. To extract its bottom-sphere coordinate, define
\begin{equation}\label{eq:theta}
 \begin{aligned}
 \theta_j: F\left(\Sigma^\infty_+BC_{p^2},\Sigma^\infty_+(BS^1)^p\right)^\wedge_p
 &\xrightarrow{\xi^{-1}}
   \bigvee_{L,\rho}(\Sigma^\infty_+BW_\rho)^\wedge_p\\
 &\xrightarrow{\operatorname{pr}_j}
   (\Sigma^\infty_+BW_{g^j\varphi})^\wedge_p
   \xrightarrow{\mathrm{collapse}}S^\wedge_p,
 \end{aligned}
\end{equation}
where \(\xi^{-1}\) is the inverse in \(\Fun(BC_p,\Sp)\) and the last arrow is induced by \(BW_{g^j\varphi}\to *\), with trivial \(C_p\)-action on \(S^\wedge_p\).

For \(\alpha\in\bigl[\Sigma^\infty_+BC_{p^2},\Sigma^\infty_+(BS^1)^p\bigr]^{C_p}\), the completed adjoint of \(\alpha\) is a morphism in \(\Fun(BC_p,\Sp)\),
\begin{equation}\label{eq:completed-adjoint}
 S^\wedge_p\longrightarrow F\left(\Sigma^\infty_+BC_{p^2},\Sigma^\infty_+(BS^1)^p\right)^\wedge_p.
\end{equation}
Write \([\alpha]\) for its class in \(\pi_0(( F\left(\Sigma^\infty_+BC_{p^2},\Sigma^\infty_+(BS^1)^p\right)^\wedge_p)^{hC_p})\). Restricting \eqref{eq:adjoint} along \(u\mapsto(u,1)\) gives \(g^j\varphi\). For \(\alpha=g^j\varphi\), composing \eqref{eq:completed-adjoint} with \(\xi^{-1}\) gives the bottom-sphere inclusion into the \(g^j\varphi\)-summand of \eqref{eq:msz}, induced by its fixed bar vertex. Its projection to the \(i\)-th summand is zero for \(i\ne j\). For \(i=j\), the remaining composite in \eqref{eq:theta} is induced by \(*\to BW_{g^j\varphi}\to *\), so it is the identity. These identities hold in \(\Fun(BC_p,\Sp)\), so the induced homotopy-fixed-point classes satisfy
\begin{equation}\label{eq:delta}
 (\theta_i^{hC_p})_*[g^j\varphi]=\delta_{ij}\cdot1,
 \qquad 0\le i,j<p.
\end{equation}

The Segal conjecture for \(C_p\) \cite{Lin, Gunawardena, Carlsson}
\begin{equation}\label{eq:segal}
 A(C_p)^\wedge_p\xrightarrow{\cong}\pi_0((S^\wedge_p)^{hC_p}).
\end{equation}
enables us to track the Burnside factor in \(\Psi_X\). Naturality of the sphere action and adjunction identify scalar multiplications of elements in \(A(C_p)\) with precomposition on \(S^\wedge_p\), which commutes with postcomposition by \(\theta_j\). Under \eqref{eq:segal}, the coordinate maps \(\alpha\mapsto(\theta_j^{hC_p})_*[\alpha]\) are \(A(C_p)\)-linear.

The homomorphism \(\Phi\) in \eqref{eq:burnside-detection} extends \(p\)-adically to
\[
 \Phi:A(C_p)^\wedge_p=\mathbb Z_p\{1,[C_p/e]\}
 \longrightarrow\mathbb Z_p,
 \qquad x+y[C_p/e]\longmapsto x.
\]
Applying \(\Phi\) to these coordinates gives the additive homomorphism
\begin{equation*}
 \begin{aligned}
 \Xi:\bigl[\Sigma^\infty_+BC_{p^2},\Sigma^\infty_+(BS^1)^p\bigr]^{C_p}&\longrightarrow
              \mathbb Z_p[C_p]=\mathbb Z_p[g]/(g^p-1),\\
 \alpha&\longmapsto\sum_{j=0}^{p-1}
       \Phi \left((\theta_j^{hC_p})_*[\alpha]\right)g^j.
 \end{aligned}
\end{equation*}
The map \(\Xi\) is \(A(C_p)\)-linear for the action on its target induced by \(\Phi\), and \eqref{eq:delta} gives \(\Xi(g^i\varphi)=g^i\) for \(0\le i<p\).

Using \(g^p=1\) and \(\Phi(p-[C_p/e])=p\), we obtain
\[
\begin{aligned}
\Xi(p^m \Psi_X^n\varphi)&=\Xi\bigl(p^m\cdot (p-[C_p/e])^n(1-g)^n \varphi \bigr) \\
&=p^{m+n}\Xi((1-g)^n\varphi)\\
 &=p^{m+n}(1-g)^n.
\end{aligned}
\]
Since the right-hand side is nonzero, \(p^m \Psi_X^n\varphi\ne0\), so \(p^m \Psi_X^n\ne0\).
\end{proof}

\subsection{Finite stages}\label{subsec:finite-stages}

We now pass from $\Sigma^\infty_+ (BS^1)^p$ to finite spectra. Using the standard model \(BS^1=\mathbb CP^\infty\), the coordinate inclusions give
\[
 \Sigma^\infty_+(BS^1)^p\simeq\operatorname*{hocolim}_{k\ge0}\Sigma^\infty_+(\mathbb CP^k)^p.
\]
These inclusions are \(C_p\)-equivariant and commute with \(\Psi\). We show that restriction along them detects every nonzero \(p^m \Psi^n\) supplied by Theorem~\ref{thm:nonvanishing}.

\begin{theorem}\label{thm:finite-stages}
For every \(m\ge0\) and \(n\ge1\), the ghost \(\Psi\) on \(\Sigma^\infty_+(\mathbb CP^k)^p\) satisfies \(p^m \Psi^n\ne0\) for all sufficiently large \(k\).
\end{theorem}

\begin{proof}
The Milnor exact sequence for the displayed homotopy colimit is
\[
\begin{aligned}
0\longrightarrow {}&
 \underset{k}{\varprojlim}{}^{1}
 \bigl[\Sigma\Sigma^\infty_+(\mathbb CP^k)^p,
       \Sigma^\infty_+(BS^1)^p\bigr]^{C_p}\\
\overset{j}{\longrightarrow} {}&
 \bigl[\Sigma^\infty_+(BS^1)^p,
       \Sigma^\infty_+(BS^1)^p\bigr]^{C_p}\\
\overset{i}{\longrightarrow} {}&
 \varprojlim_k
 \bigl[\Sigma^\infty_+(\mathbb CP^k)^p,
       \Sigma^\infty_+(BS^1)^p\bigr]^{C_p}
 \longrightarrow0.
\end{aligned}
\]
{To show that $i$ is injective, it suffices to show that the groups \(\bigl[\Sigma\Sigma^\infty_+(\mathbb CP^k)^p,
       \Sigma^\infty_+(BS^1)^p\bigr]^{C_p}\) are finitely generated and torsion for all $k$. An inverse system of finite groups satisfies the Mittag--Leffler condition, so its \(\varprojlim\nolimits^{1}\) vanishes.}

Using the tom Dieck splitting \cite[Theorem 4.1]{MSZ}, $\pi_n^{C_p}(\Sigma^\infty_+(BS^1)^p)$ and $\pi_n^e(\Sigma^\infty_+(BS^1)^p)$ are finite direct sums of some classical stable homotopy groups of spaces of finite type, so these homotopy groups are finitely generated. Since \(\Sigma\Sigma^\infty_+(\mathbb CP^k)^p\) is a finite \(C_p\)-spectrum, induction over its orbit cells shows that the mapping groups are finitely generated as well.

It remains to prove torsion. Since each \(\Sigma\Sigma^\infty_+(\mathbb CP^k)^p\) is finite, consider tensoring the mapping group with $\mathbb Q$. The underlying spaces \((\mathbb CP^k)^p\), \((BS^1)^p\) and their fixed-point diagonals \(\mathbb CP^k\), \(BS^1\) all have {rational homology concentrated in even degrees. By the rational equivariant splitting~\cite[Appendix~A, Theorem~A.4]{GM}, the extra suspension on \(\Sigma^\infty_+(\mathbb CP^k)^p\) makes both classical rational mapping groups zero.} It follows that the mapping groups \(\bigl[\Sigma\Sigma^\infty_+(\mathbb CP^k)^p,
       \Sigma^\infty_+(BS^1)^p\bigr]^{C_p}\) are torsion.

This proves that $i$ is injective. By Theorem~\ref{thm:nonvanishing}, for each fixed \(m\ge0\) and \(n\ge1\), there is an integer \(k_0\ge0\) such that the composite
\[
 \Sigma^\infty_+(\mathbb CP^{k_0})^p
 \xrightarrow{\mathrm{inc}}\Sigma^\infty_+(BS^1)^p
 \xrightarrow{p^m \Psi^n}\Sigma^\infty_+(BS^1)^p
\]
is nonzero. For every \(k\ge k_0\), this composite factors through the self-map \(p^m \Psi^n\) of \(\Sigma^\infty_+(\mathbb CP^k)^p\), since the inclusions commute with \(\Psi\). This self-map is nonzero for every \(k\ge k_0\).
\end{proof}

Applying Theorem~\ref{lem:induction}, these finite examples can be extended to arbitrary finite groups while preserving their nonzero powers.

\begin{corollary}\label{cor:finite-groups}\label{cor:long-refinements}
Let \(G\) be a finite group containing \(C_p\). On \(\Ind_{C_p}^G\Sigma^\infty_+(\mathbb CP^k)^p\), the ghost \(\Ind_{C_p}^G \Psi\) {has zero underlying map. For every fixed \(m\ge0\) and \(n\ge1\), we have}
\[
 p^m\bigl(\Ind_{C_p}^G \Psi\bigr)^n\ne0
\]
for all sufficiently large \(k\). Consequently every nontrivial finite group admits ghosts with arbitrarily long nonzero powers on finite spectra.
\end{corollary}

\begin{proof}
Corollary~\ref{cor:induced-ghost} gives the ghost and underlying vanishing properties. Since
\[
 \Ind_{C_p}^G(p^m \Psi^n)
 =p^m\bigl(\Ind_{C_p}^G \Psi\bigr)^n,
\]
Theorem~\ref{thm:finite-stages} implies the ghost is nontrivial on some finite subcomplexes. The claim follows from the induction in Theorem~\ref{lem:induction}.
\end{proof}

Corollary~\ref{cor:finite-groups} proves the integral nonvanishing required in Theorem~\ref{thm:long-main}. We complete that theorem by proving finite \(p\)-power additive order and nonvanishing after \(p\)-localization.

\begin{proof}[Proof of Theorem~\ref{thm:long-main}]
Fix \(m\ge0\) and \(n\ge1\). Choose \(C_p\le G\) and a sufficiently large \(k\) as in Corollary~\ref{cor:finite-groups}. The ghost lemma~\cite[Theorem~3.5]{Christensen} implies that \(\Psi\) is nilpotent in the $C_p$-equivariant homotopy category. Since \(g^p=1\), 
\[p\cdot (1-g)=\sum_{j=2}^{p}(-1)^j \binom{p}{j}(1-g)^j\]
Since $[C_p/e]^2=p\cdot [C_p/e]$, we obtain
\begin{equation}\label{eq:proofthmC}
    p^p \Psi =p^p(p-[C_p/e])(1-g)=\sum_{j=2}^p(-1)^j\binom pj p^{p-j}\Psi^j.
\end{equation}

Choose $d\geq 1$ such that $\Psi^d=0$, and consider the polynomial 
\[Q(x)=\sum_{j=2}^p(-1)^j\binom pj p^{p-j}x^{j-2}.\]
Thus, \eqref{eq:proofthmC} becomes $p^p\Psi=\Psi^2Q(\Psi)$.
Iterating it gives
\[
  p^{p(d-1)}\Psi
  =\Psi^dQ(\Psi)^{d-1}
  =0.
\]
We conclude that $\Psi$, and hence every power of $\Psi$, is $p$-power torsion. This additive order is preserved by induction by Theorem~\ref{lem:induction}. 
\end{proof}

\section{Nonfullness and failure of classification}\label{sec:module-detection}

In this section, we use the constructions and results of Section~\ref{sec:long-ghosts} to prove Theorem~\ref{thm:main-results}\textup{(\ref{item:main-nonfullness})} and Theorem~\ref{cor:same-modules}.

We consider the presheaf formulation of Bohmann--May \cite[Section~1]{BohmannMay}. In the integral or $p$-local $G$-stable homotopy category, let
\begin{equation}\label{eq:coefficient-categories}
\begin{aligned}
 \mathcal C_{\Z}&=\operatorname{add}
   \{G/H_+\wedge S^n:H\leq G,\ n\in\Z\},\\
 \mathcal C_{\RO}&=\operatorname{add}
   \{G/H_+\wedge S^V:H\leq G,\ V\in\RO(G)\},\\
 \mathcal C_{\mathrm{all}}&=\operatorname{add}
   \{G_+\wedge_H S^W:H\leq G,\ W\in\RO(H)\}.
\end{aligned}
\end{equation}
Here $\operatorname{add}$ denotes the skeleton of the full subcategory of finite direct sums of the corresponding spectra. For any of the above categories, define
\[\operatorname{Mod}\mathcal C
   =\Fun^\oplus(\mathcal C^{\mathrm{op}},\mathrm{Ab})\]
and
\begin{align*}
    h_{\mathcal C}:\Ho(\Sp^G) &\to \operatorname{Mod}\mathcal C\\
     X &\mapsto [-,X]^G.
\end{align*}
When $\mathcal C$ is clear from context, we abbreviate $h=h_{\mathcal C}$. In particular, a map $f$ is a $\mathcal C$-ghost if $h(f)=0$. Moreover, $f$ is an equivalence if and only if $h(f)$ is an equivalence, by the equivariant Whitehead theorem \cite[Theorem~III.4.2]{MM}.

For any $\mathcal C$-ghost $w:X\to Z$, a distinguished triangle
\[
 X\xrightarrow{w}Z\xrightarrow{i}E\xrightarrow{\delta}\Sigma X
\]
gives a short exact sequence
\begin{equation}\label{eq:double-ghost-module-sequence}
 0\longrightarrow h(Z)\xrightarrow{h(i)}h(E)
 \xrightarrow{h(\delta)}h(\Sigma X)\longrightarrow0
\end{equation}
because $h(w)=h(\Sigma w)=0$. The following lemma makes the splitting in Pirashvili--Redondo~\cite[Theorem~1]{PR} explicit when $w$ is a composite of two ghosts.

\begin{lemma}\label{lem:double-ghost-splitting}
Let $u:X\to Y$ and $v:Y\to Z$ be $\mathcal C$-ghosts, and put $w=vu$. Then the short exact sequence \eqref{eq:double-ghost-module-sequence} associated to a distinguished triangle completing $w$ splits in $\operatorname{Mod}\mathcal C$.
\end{lemma}
\begin{proof}
Complete $u$ to a distinguished triangle. Since $w=vu$, the identity of $X$ and the map $v$ extend to a morphism of triangles
\[
\begin{tikzcd}[column sep=large, row sep=large]
 X \arrow[r,"u"] \arrow[d,equal]
   & Y \arrow[r,"j"] \arrow[d,"v"]
   & D \arrow[r,"q"] \arrow[d,"t"]
   & \Sigma X \arrow[d,equal] \\
 X \arrow[r,"w"']
   & Z \arrow[r,"i"']
   & E \arrow[r,"\delta"']
   & \Sigma X .
\end{tikzcd}
\]
Applying $h$ gives the solid arrows in the following commutative diagram with exact rows.
\[
\begin{tikzcd}[column sep=large, row sep=huge]
 0 \arrow[r]
   & h(Y) \arrow[r,"h(j)"] \arrow[d,"h(v)=0"']
   & h(D) \arrow[r,"h(q)"] \arrow[d,"h(t)"]
   & h(\Sigma X) \arrow[r] \arrow[d,equal]
       \arrow[dl,dashed,"s"']
   & 0 \\
 0 \arrow[r]
   & h(Z) \arrow[r,"h(i)"']
   & h(E) \arrow[r,"h(\delta)"']
   & h(\Sigma X) \arrow[r]
   & 0 .
\end{tikzcd}
\]
Exactness and the relations $h(t)h(j)=h(i)h(v)=0$ imply that there exists the dashed module map $s$ with $s\,h(q)=h(t)$. We then have $h(\delta)s\,h(q)=h(q)$, and thus $h(\delta)s=\Id$ by the surjectivity of $h(q)$.
\end{proof}

We are now ready to prove the claim of nonfullness in Theorem~\ref{thm:main-results}\textup{(\ref{item:main-nonfullness})}. Recall that fullness on finite spectra means that the natural map
\begin{equation}\label{eq:module-comparison}
 [X,Y]^G\longrightarrow
 \Hom_{\operatorname{Mod}\mathcal C}(h(X),h(Y))
\end{equation}
is surjective for every pair of finite spectra $X,Y$.

\begin{proof}[Proof of Theorem~\ref{thm:main-results}\textup{(\ref{item:main-nonfullness})}]
Corollary~\ref{cor:finite-groups} gives a finite integral $G$-spectrum $Z$ with a representation-orbit ghost $\Psi_Z:Z\to Z$ such that $\Psi_Z^2\ne0$. Theorem~\ref{thm:long-main} gives such a pair in the $p$-local category as well. In either category, choose a distinguished triangle
\[
 Z\xrightarrow{\Psi_Z^2}Z\xrightarrow{i}E\xrightarrow{\delta}\Sigma Z.
\]
For any $\mathcal C$ in \eqref{eq:coefficient-categories}, Lemma~\ref{lem:double-ghost-splitting} gives a section $s:h(\Sigma Z)\to h(E)$ of $h(\delta)$. To prove nonfullness, it suffices to show that $s$ is not induced by a stable map. If $s$ were induced by a stable map $r:\Sigma Z\to E$, then $h(\delta r)=\Id$, so $\delta r$ would be invertible because $h$ reflects isomorphisms. But $(\Sigma \Psi_Z^2)\delta r=0$, forcing $\Psi_Z^2=0$, a contradiction.
\end{proof}

In the notation of Lemma~\ref{lem:double-ghost-splitting}, the splitting already gives $h(E)\cong h(Z\vee\Sigma X)$. For Theorem~\ref{cor:same-modules}, we must also arrange that $E\not\simeq Z\vee\Sigma X$. We will do so by constructing a nonzero double ghost on a finite spectrum annihilated by a power of $p$ and comparing the orders of the resulting finite mapping groups.

\begin{proof}[Proof of Theorem~\ref{cor:same-modules}]
Work integrally and choose a subgroup $C_p\leq G$. Theorem~\ref{thm:finite-stages} and the proof of Theorem~\ref{thm:long-main} give a finite integral $C_p$-spectrum $Y$ on which the natural ghost of Proposition~\ref{prop:natural-ghost} satisfies
\[
 \Psi_Y^3\ne0,\qquad p^r \Psi_Y=0
\]
for some $r\geq1$. Set $Z=Y\wedge S/p^r$, which fits into the distinguished triangle
\[
 Y\xrightarrow{p^r}Y\xrightarrow{j}Z
   \xrightarrow{\partial}\Sigma Y.
\]
We claim that $j\Psi_Y^2\ne0$. Otherwise, exactness of $[Y,-]^{C_p}$ would give $\Psi_Y^2=p^r v$ for some $v:Y\to Y$, which implies
\[
 \Psi_Y^3=(p^r v)\Psi_Y=v(p^r \Psi_Y)=0,
\]
a contradiction. Naturality gives
\[
 \Psi_Z^2j=j\Psi_Y^2\ne0,
\]
so the ghost $\Psi_Z$ on $Z$ also has nonzero square.

Since $(p^r\Id_Z)j=jp^r=0$, exactness shows that $p^r\Id_Z$ factors through $\partial$. Multiplying this factorization by $p^r$ and using $(\Sigma p^r)\partial=0$ gives $p^{2r}\Id_Z=0$. By the induction principle in Theorem~\ref{lem:induction}, the same properties hold after we induce the pair $(Z,\Psi_Z)$ to $G$. Abusing the notation, we also denote the induced pair by $(Z,\Psi_Z)$.

Complete $\Psi_Z^2$ to a distinguished triangle
\begin{equation}\label{eq:classification-cofiber}
 Z\xrightarrow{\Psi_Z^2}Z\xrightarrow{i}A\xrightarrow{\delta}\Sigma Z
\end{equation}
and set $B=Z\vee\Sigma Z$. Both $A$ and $B$ are finite and $p$-local, since cofibers, suspensions, and finite wedges preserve these properties.

Take $h=h_{\mathcal C_{\mathrm{all}}}$. Lemma~\ref{lem:double-ghost-splitting} gives
\[
 h(A)\cong h(Z)\oplus h(\Sigma Z)\cong h(B).
\]
Evaluating this module isomorphism at $G_+\wedge_L S^{\Res_L^H W}$ for $L\leq H$, we have isomorphisms
\[
 \underline{\pi}_W^H(A)\cong\underline{\pi}_W^H(B)
 \qquad(H\leq G,\ W\in\RO(H))
\]
compatible with all stable operations.

We now show that $A\not\simeq B$ by comparing the orders of $[Z,A]^G$ and $[Z,B]^G$. The tom Dieck splitting~\cite[Theorem~V.11.1]{LMS} expresses the integer-graded homotopy groups of orbit spectra as finite sums of classical stable homotopy groups of spaces of finite type, which are finitely generated. Therefore, after induction, the mapping groups between finite integral $G$-spectra are finitely generated as well. Since $p^{2r}\Id_Z=0$, all mapping groups below are also annihilated by $p^{2r}$, so they are finite.

Applying $[Z,-]^G$ to \eqref{eq:classification-cofiber} gives
\[
 [Z,Z]^G\xrightarrow{(\Psi_Z^2)_*}[Z,Z]^G\longrightarrow[Z,A]^G
 \longrightarrow[Z,\Sigma Z]^G
 \xrightarrow{(\Sigma \Psi_Z^2)_*}[Z,\Sigma Z]^G.
\]
Consequently,
\[
 \begin{aligned}
 |[Z,A]^G|
   &=|\operatorname{coker}(\Psi_Z^2)_*|\,|\ker(\Sigma \Psi_Z^2)_*|\\
   &<|[Z,Z]^G|\,|[Z,\Sigma Z]^G|=|[Z,B]^G|.
 \end{aligned}
\]
The strict inequality holds because $(\Psi_Z^2)_*(\Id_Z)=\Psi_Z^2\ne0$. This proves that $A\not\simeq B$.
\end{proof}

\section{Homological consequences}\label{sec:homological}

We prove Theorem~\ref{cor:homological} by applying homological dimension bounds to the ghost powers of Theorem~\ref{thm:long-main}. Throughout this section we work $p$-locally. Let $\mathcal C$ be one of the categories in \eqref{eq:coefficient-categories}, and write $h=h_{\mathcal C}$.

Recall that $M\in \operatorname{Mod}_{\mathcal C}$ is flat if $M\otimes_{\mathcal C}-$ is exact. The projective dimension $\operatorname{pd}_{\mathcal C}M$ and flat dimension $\operatorname{fd}_{\mathcal C}M$ are the least lengths of projective and flat resolutions of $M$, respectively. Dually, the injective dimension $\operatorname{id}_{\mathcal C}M$ is the least length of an injective resolution. The weak global and global dimensions are
\[
 \begin{aligned}
 \operatorname{w.gl.dim}(\operatorname{Mod}\mathcal C)
   &=\sup_M\operatorname{fd}_{\mathcal C}M,\\
 \operatorname{gl.dim}(\operatorname{Mod}\mathcal C)
   &=\sup_M\operatorname{pd}_{\mathcal C}M
    =\sup_M\operatorname{id}_{\mathcal C}M,
 \end{aligned}
\]
for all $M\in \operatorname{Mod}_{\mathcal C}$ \cite{Weibel}.

\begin{proof}[Proof of Theorem~\ref{cor:homological}]
We use a coefficient-category version of the bounds of Christensen~\cite{Christensen} and Hovey--Lockridge~\cite{HL-ghost,HL-dimensions}, recalled for graded rings in Ma--Xu~\cite[Proposition~6.2]{MX}. If a composite of $n\geq1$ $\mathcal C$-ghosts from a finite $G$-spectrum $X$ to a $G$-spectrum $Y$ is nonzero, then
\[
 \operatorname{fd}_{\mathcal C}h(X),\qquad
 \operatorname{pd}_{\mathcal C}h(X),\qquad
 \operatorname{id}_{\mathcal C}h(Y)\geq n.
\]
Each $\mathcal C$ is small, additive, and closed under integer suspensions, and its objects are compact and generate the ambient category. If $h(F)$ is flat, co-Yoneda and thick induction show that every map from a finite spectrum to $F$ factors through a finite sum of objects of $\mathcal C$. Every ghost out of $F$ is therefore phantom. If $\operatorname{fd}_{\mathcal C}h(X)<n$, the $(n-1)$st stage of a universal ghost tower has flat homotopy module, up to suspension. Every composite of $n$ ghosts out of $X$ factors through a ghost out of that stage and is zero when $X$ is finite. This proves the flat bound, and $\operatorname{fd}\leq\operatorname{pd}$ gives the projective bound. The injective bound follows by dual dimension shifting, using Brown representability to realize injective $\mathcal C$-modules.

For every $n\geq1$, Theorem~\ref{thm:long-main} with $m=0$ gives a finite $p$-local $G$-spectrum $Y_n$ and an endomorphism $f$ that is a ghost for all three families, with $f^n\ne0$. Applying the bound with $X=Y=Y_n$ to each category in \eqref{eq:coefficient-categories} proves that all three homotopy modules have flat, projective, and injective dimension at least $n$. Taking suprema over $n$ shows that all three module categories have infinite weak global dimension and infinite global dimension.
\end{proof}

\begin{remark}
The four-cell examples of Theorem~\ref{thm:main} already give non-injective homotopy modules. Theorem~\ref{cor:homological} further shows that projective and injective dimensions are unbounded among the homotopy modules of finite spectra. It does not assert that a single finite spectrum has a homotopy module of infinite injective dimension.
\end{remark}

\begin{remark}
{Infinite ghost dimension of the ordinary $p$-local sphere is already known \cite[Proposition~2.1 and Corollary~2.2]{HL-dimensions}. For every $n$, this gives a nonzero composite of $n$ ghosts with finite source, although the intermediate spectra and the target need not be finite. Induction from the trivial subgroup preserves nonvanishing and produces ghosts for all three families in \eqref{eq:coefficient-categories}. The bounds above already give unbounded flat and projective dimensions among the homotopy modules of finite $G$-spectra, and infinite weak global and global dimensions for all three coefficient categories. Our construction additionally gives arbitrarily long nonzero powers of a ghost endomorphism of a finite spectrum. These powers provide simultaneous lower bounds on the flat, projective, and injective dimensions of all three homotopy modules of the same finite spectrum.}
\end{remark}

\raggedbottom

\bibliographystyle{alpha}
\bibliography{ref}
\end{document}